\documentclass[11pt]{article}

\newcommand{\hf}{\textstyle{\frac{1}{2}}}
\usepackage{natbib,amsmath,amssymb,amsmath,amsthm,amsfonts}
\usepackage{graphicx,url}
\title{Sparse Principal Components Analysis}
\author{Iain M. Johnstone and Arthur Yu Lu \\
Stanford University and Renaissance Technologies}

\date{January 1, 2004}

\newtheorem{lemma}{Lemma}
\newtheorem{theorem}{Theorem}
\newtheorem{proposition}{Proposition}

\begin{document}
\maketitle
\begin{center}
\textbf{Extended Abstract}
\end{center}
\small Principal components analysis (PCA) is a classical method for
the reduction of dimensionality of data in the form of $n$
observations (or cases) of a vector with $p$ variables. Contemporary
data sets often have $p$ comparable to, or even much larger than $n$.
Our main assertions, in such settings, are (a) that some initial
reduction in dimensionality is desirable before applying any PCA-type
search for principal modes, and (b) the initial reduction in
dimensionality is best achieved by working in a basis in which the
signals have a sparse representation.  We describe a simple asymptotic
model in which the estimate of the leading principal component vector
via standard PCA is consistent if and only if $p(n)/n \rightarrow 0$.
We provide a simple algorithm for selecting a subset of coordinates
with largest sample variances, and show that if PCA is done on the
selected subset, then consistency is recovered, even if $p(n) \gg n$.

Our main setting is that of signals and images, in which the number of
sampling points, or pixels, is often comparable with or larger than
the number of cases, $n$. Our particular example here is the
electrocardiogram (ECG) signal of the beating heart, but similar approaches
have been used, say, for PCA on libraries of face images.

Standard PCA involves an $O( \min (p^3, n^3))$ search for directions
of maximum variance. But if we have some \textit{a priori} way of selecting $k
\ll \min(n,p)$ coordinates in which most of the variation among cases
is to be found, then the complexity of PCA is much reduced, to $O(k^3)$.
This is a computational reason, but if there is instrumental or other
observational noise in each case that is uncorrelated with or
independent of relevant case-to-case variation, then there is another
compelling reason to preselect a small subset of variables before running PCA.

Indeed, we construct a model of factor analysis type and show that 
ordinary PCA can produce a consistent (as $n \rightarrow \infty$)
estimate of the principal factor if and only if $p(n)$ is
asymptotically of smaller order than $n$.  Heuristically, if $p(n) \geq c n$,
there is so much observational noise and so many dimensions over which
to search, that a spurious noise maximum will always drown out the
true factor.

Fortunately, it is often reasonable to expect such small subsets of
variables to exist: Much recent research in signal and image analysis
has sought orthonormal basis and related systems in which typical
signals have \textit{sparse} representations: most co-ordinates
have small signal energies.  If such a basis is used to
represent a signal -- we use wavelets as the classical example here --
then the variation in many coordinates is likely to be very small.

Consequently, we study a simple ``sparse PCA'' algorithm with
the following ingredients:
a) given a suitable orthobasis, compute coefficients for each case,
b) compute sample variances (over cases) for each coordinate
in the basis, and select the $k$ coordinates of largest sample variance,
c) run standard PCA on the selected $k$ coordinates, obtaining up
to $k$ estimated eigenvectors,
d) if desired, use soft or hard thresholding to denoise these estimated
eigenvectors, and
e) re-express the (denoised) sparse PCA eigenvector estimates in the
original signal domain.

We illustrate the algorithm on some exercise ECG data, and also
develop theory to show 
in a single factor model, under an appropriate sparsity
assumption, that it indeed overcomes the inconsistency problems when $p(n) \geq
cn,$ and yields consistent estimates of the principal factor.

\normalsize


\section{Introduction}
\label{sec:introduction}

Suppose $\{ x_i, i = 1, \ldots, n \}$ is a dataset of $n$ observations
on $p$ variables. Standard principal components analysis (PCA) looks
for vectors $\xi$ that maximize
\begin{equation}
  \label{eq:pcacrit}
  \mbox{Var} \, (\xi^T x_i ) / \| \xi \|^2.
\end{equation}
If $\xi_1, \ldots , \xi_k$ have already been found by this
optimization, then the maximum defining $\xi_{k+1}$ is taken over
vectors $\xi$ orthogonal to $\xi_1, \ldots, \xi_k$.

Our interest lies in situations in which each $x_i$ is a realization
of a possibly high dimensional signal, so that $p$ is comparable in
magnitude to $n$, or may even be larger. 
In addition, we have in mind settings in which the signals $x_i$
contain localized features, so that the principal modes of variation
sought by PCA may well be localized also.

Consider, for example, the sample of an electrocardiogram (ECG) 
in Figure \ref{fig:ecgsamp} showing some 13 consecutive heart beat cycles as
recorded by one of the standard ECG electrodes.
Individual beats are notable for features such as the sharp spike
(``QRS complex'') and the subsequent lower peak (``T wave''),
shown schematically in the second panel.
The presence of these local features, of differing spatial scales,
suggests the use of wavelet bases for efficient representation.
Traditional ECG analysis focuses on averages of a series of beats.  If
one were to look instead at beat to beat \textit{variation}, one
might expect these local features to play a significant role in
the principal component eigenvectors.

\begin{figure*}[htb]
     \begin{center}
       \leavevmode
\centerline{\includegraphics[ width = .8\textwidth
,angle = 0]{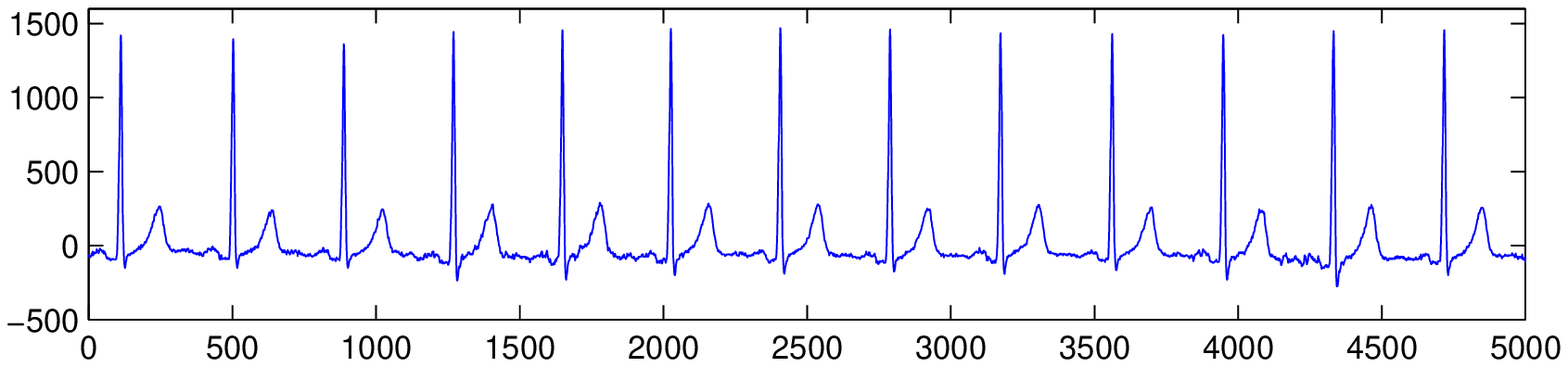}}
\centerline{\includegraphics[ height = 1.5in]{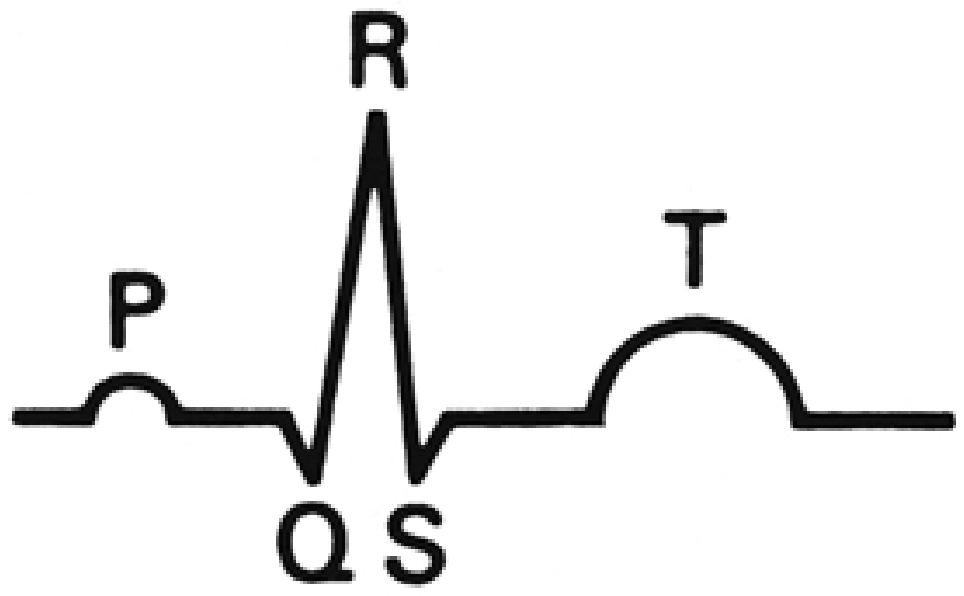}}
\caption{ (a) Sample of thirteen beats from one electrode of an
  electrocardiogram taken in the laboratory of Victor Froelicher, MD,
  Palo Alto VA.
(b) Cartoon of the key features of the cardiac cycle reflected in the
ECG trace, from \cite{hamp97}.}
       \label{fig:ecgsamp}
     \end{center}
\end{figure*}


Returning to the general situation, the main contentions of this paper
are:

(a) that when $p$ is comparable to $n$, some reduction
in dimensionality is desirable before applying any PCA-type search for
principal modes, and

(b) the reduction in dimensionality is best achieved by working in a
basis in which the signals have a sparse representation.

We will support these assertions with arguments based on statistical
performance and computational cost.

We begin, however, with an illustration of our results on
a simple constructed example.
Consider a single component (or single factor) model, in which, when
viewed as $p-$dimensional column vectors
\begin{equation}
  \label{eq:singlecpt}
  x_i = v_i \rho + \sigma z_i, \qquad \quad i = 1, \ldots, n
\end{equation}
in which $\rho \in \mathbb{R}^p$ is the single component to be estimated,
$v_i \sim N(0,1)$ are i.i.d. Gaussian random effects and $z_i \sim
N_p(0,I)$ are independent $p-$dimensional noise vectors.

\begin{figure*}[htb]
     \begin{center}
       \leavevmode
        \centerline{\includegraphics[ width = .8\textwidth,angle =
    -90]{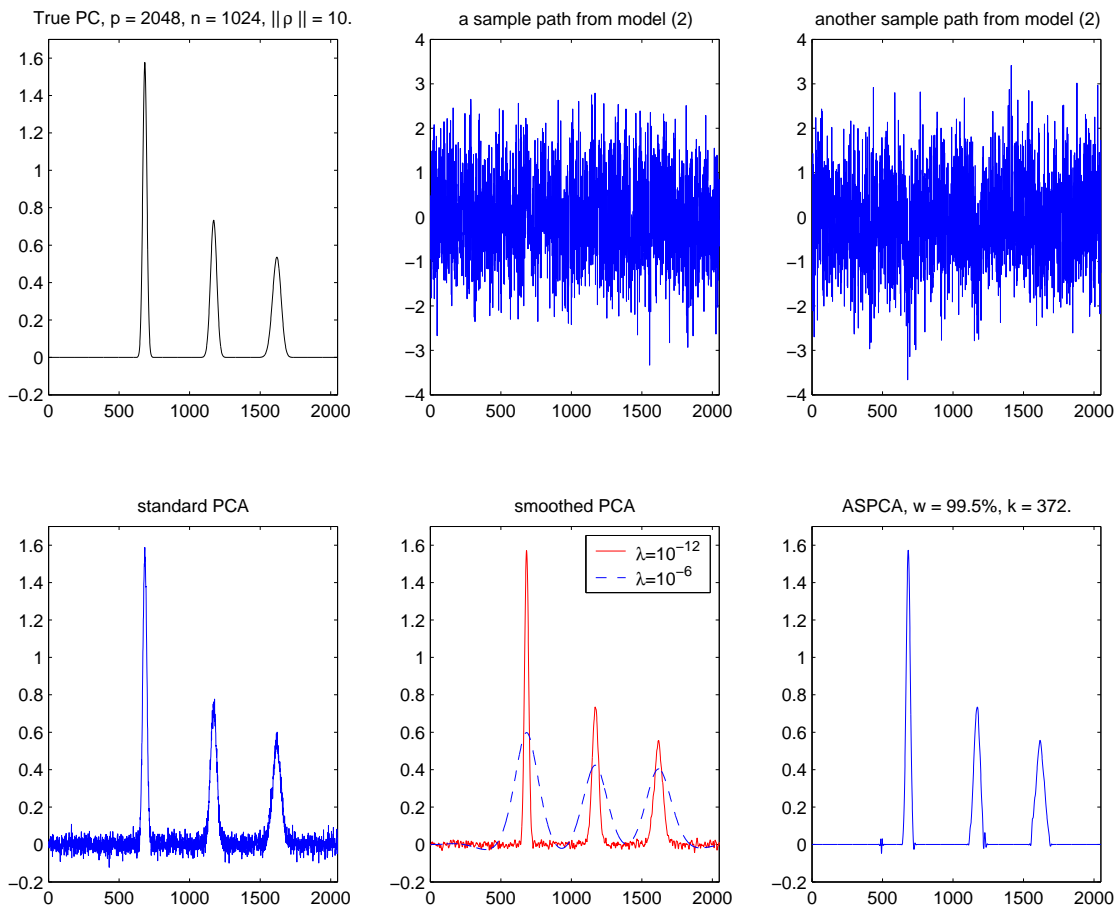}}
\caption{ \small True 
  principal component, the ``3-peak'' curve. Panel (a): the single
  component $\rho_l = f(l/n)$ where $f(t) = C \bigl\{0.7 B(1500, 3000)
  + 0.5 B(1200, 900) + 0.5 B(600, 160) \bigr\}$ and $B(a,b)(t) =
  [\Gamma(a+b)/(\Gamma(a) \Gamma(b))] t^{a-1} (1-t)^{b-1}$ denotes the
  Beta density on $[0,1]$.  
  Panels (b,c): Two sample paths drawn from model (2) with $\sigma = 1$.
  $n= 1024$ replications in total, $ p = 2048$.
  \ \ \  (d): Sample
  principal component by standard PCA. (e): Sample principal component
  by smoothed PCA using $\lambda = 10^{-12}$ and $\lambda = 10^{-6}$.
  (f): Sample principal component by  sparse PCA with
  weighting function $w = 99.5\%$, $k = 372$. }
       \label{fig:betacurves}
     \end{center}
\end{figure*}

Panel (a) of Figure \ref{fig:betacurves} shows an example of $\rho$
with $p = 2048$ and the vector $\rho_l = f(l/n)$ where $f(t)$ is a
mixture of Beta densities on $[0,1]$, scaled so that $\| \rho \| = 
(\sum_1^p \rho_l^2)^{1/2} = 10.$
Panels (b) and (c) show two sample paths from model
(\ref{eq:singlecpt}): the random effect $v_i \rho$ is hard to discern
individual cases.
Panel (d) shows the result of standard PCA applied to $n=1024$
observations from (\ref{eq:singlecpt}) with $\sigma=1$.
The effect of the noise remains clearly visible in the estimated
principal eigenvector.

For functional data of this type, a regularized approach to PCA has
been proposed by \citet{risi91} and \citet{silv96}, see 
also \citet{rasi97} and references therein.
While smoothing can be incorporated in various ways, we
illustrate the method discussed also in \citet[Ch. 7]{rasi97}, which replaces
(\ref{eq:pcacrit}) with 
\begin{equation}
  \label{eq:fpacrit}
    \mbox{Var} \, (\xi^T x_i ) / [ \| \xi \|^2 + \lambda \| D^2 \xi
    \|^2 ],
\end{equation}
where $D^2 \xi$ is the $(p-2) \times 1$ vector of second differences
of $\xi$ and $\lambda \in (0,\infty)$ is the regularization parameter.

Panel (e) shows the estimated first principal component vector found
by maximizing (\ref{eq:fpacrit}) with $\lambda = 10^{-12}$ and
$\lambda = 10^{-6}$ respectively. Neither is really satisfactory as an
estimate: the first recovers the original peak heights, but fails
fully to suppress the remaining baseline noise, while the second
grossly oversmooths the peaks in an effort to remove all trace of
noise.
Further investigation with other choices of $\lambda$ confirms the
impression already conveyed here: no single choice of $\lambda$
succeeds both in preserving peak heights and in removing baseline
noise.

Panel (f) shows the result of the adaptive sparse PCA algorithm to be
introduced below: evidently both goals are accomplished quite
satisfactorily in this example.

\section{The need to select subsets: (in)consistency of classical PCA}
\label{sec:select}

A basic element of our sparse PCA proposal is initial selection of a
relatively small subset of the initial $p$ variables before any PCA is
attempted. In this section, we formulate some (in)consistency results
that motivate this initial step.

Consider first the single component model (\ref{eq:singlecpt}). 
The presence of noise means that the sample covariance matrix
$S = n^{-1} \sum_{i=1}^n x_i x_i^T$ will typically have $\min (n,p)$
non-zero eigenvalues.
Let $\hat \rho$ be the unit eigenvector associated with the largest
sample eigenvalue---with probability one it is uniquely determined up
to sign.

One natural measure of the closeness of $\hat \rho$ to $\rho$ uses the
angle $\angle (\hat \rho, \rho)$ between the two vectors.
We decree that the signs of $\hat \rho$ and $\rho$ be taken so that
$\angle ( \hat \rho, \rho)$ lies in $[0, \pi/2]$.
It will be convenient to phrase the results in terms of an equivalent
distance measure
\begin{equation}
  \label{eq:distdef}
  \text{dist}(\hat \rho, \rho) 
  = \sin \angle ( \hat \rho, \rho ) 
  = \sqrt{1 - (\rho^T \hat \rho)^2}.
\end{equation}

For asymptotic results, we will assume that there is a sequence of
models (\ref{eq:singlecpt}) indexed by $n$. 
Thus, we allow $p(n)$ and $\rho(n)$ to depend by $n$, though the
dependence will usually not be shown explicitly.
[Of course $\sigma$ might also be allowed to vary with $n$, but for
simplicity it is assumed fixed.]

Our first interest is whether the estimate $\hat \rho$ is
consistent as $n \rightarrow \infty$. This turns out to depend
crucially on the limiting value
\begin{equation}
  \label{eq:clim}
  \lim_{n \rightarrow \infty} p(n)/n = c.
\end{equation}
We will also assume that
\begin{equation}
  \label{eq:rholim}
  \lim_{n \rightarrow \infty} \| \rho(n) \| = \varrho > 0.
\end{equation}
One setting in which this last assumption may be reasonable is when
$p(n)$ grows by adding finer scale wavelet coefficients of a fixed
function as $n$ increases.

\begin{theorem}
  \label{thm:upperbd}
Assume model (\ref{eq:singlecpt}), (\ref{eq:clim}) and
(\ref{eq:rholim}). Define
\begin{displaymath}
  \zeta(\tau;c) = \frac{4 \sqrt c}{\tau} \Bigl( 1 + \frac{2 + \sqrt
  c}{\tau} \Bigr).
\end{displaymath}
Then with probability one as $n \rightarrow \infty$,
\begin{equation}
  \label{eq:upperbd}
  \limsup_{n \rightarrow \infty} 
      \ \sin \angle ( \hat \rho, \rho ) \leq \zeta( \varrho/\sigma, c),
\end{equation}
so long as the right side is at most one.
\end{theorem}

For the proof, see Appendix \ref{sec:upper-bounds:-proof}. 
The bound $\zeta(\tau;c)$ is decreasing in the ``signal-to-noise'' ratio $\tau
= \varrho/\sigma$ and increasing in the dimension-to-sample size ratio
$c = \lim p/n$.
It approaches $0$ as $c \rightarrow 0$, and in particular it follows
that $\hat \rho$ is consistent if $p/n \rightarrow 0$.

The proof is based on an almost sure bound for
eigenvectors of perturbed symmetric matrices. It appears to give 
the correct order of convergence: in the case $p/n \rightarrow 0$, we have
\begin{displaymath}
  \zeta(\tau,p/n) \sim c(\tau) \sqrt{ p/n},
\end{displaymath}
with $c(\tau) = 4 \tau^{-1} + 8 \tau^{-2}$, 
and examination of the proof shows that in fact
\begin{displaymath}
  \angle ( \hat \rho, \rho) = O_p( \sqrt{ p/n} )
\end{displaymath}
which is consistent with the $n^{-1/2}$ convergence rate that is
typical when $p$ is fixed.

However if $c > 0$, the upper bound (\ref{eq:upperbd}) is strictly
positive.  And it turns out that $\hat \rho$ must be an inconsistent
estimate in this setting:

\begin{theorem}
  \label{th:lowerbd}
 Assume model (\ref{eq:singlecpt}), (\ref{eq:clim}) and
(\ref{eq:rholim}). If $p/n \rightarrow c > 0$, then $\hat \rho$ is
inconsistent:
\begin{displaymath}
  \liminf_{n \rightarrow \infty} E \angle ( \hat \rho, \rho ) > 0.
\end{displaymath}
\end{theorem}

In short, $\hat \rho$ is a consistent estimate of $\rho$ if and only
if $p = o(n)$.
The noise does not average out if there are too many dimensions $p$
relative to sample size $n$. A heuristic explanation for this
phenomenon is given just before the proof in Appendix
\ref{sec:lower-bounds:-proof}.

\bigskip

The inconsistency criterion extends to a considerably more general
\textit{multi-component} model.
Assume that we have $n$ curves $x_i$, observed at $p$ time
points. Viewed as $p$ dimensional column vectors, this model assumes that
\begin{equation}
x_i = \mu + \sum_{j = 1}^m v_i^j \rho^j + \sigma z_i, \quad
i = 1,\cdots,n.
\label{eq:mcm}
\end{equation}
Here $\mu$ is the mean function, which is assumed known, and hence is taken
to be zero. We make the following assumptions:

(a) The $\rho^j, j = 1, ... ,m \leq p$ are unknown, mutually
orthogonal principal components, with norms $\rho_j(n)= \| \rho^j \|$
\begin{equation}
  \label{eq:strict}
  \| \rho^1 \| > \| \rho^2 \| \geq \cdots \geq \|\rho^m \|.
\end{equation}

(b) The multipliers $v_i^j \sim N(0,1)$ are all independent
over $j = 1, \ldots, m$ and $i = 1, \ldots ,m$.

(c) The noise vectors $z_i \sim N_p(0,I)$ are independent
among themselves and also of the random effects $\{ v_i^j \}$.

For asymptotics, we add

(d) We assume that $p(n), m(n)$ and $\{ \rho^j(n), j = 1,
\ldots, m\}$ are functions of $n$, though this will generally not be
shown explicitly. We assume that the norms of the $n^{th}$ principal
components converge as sequences in $\ell_1(\mathbb{N})$:
\begin{equation}
  \label{eq:rhocge}
  \begin{split}
   & \varrho(n) = ( \| \rho^1(n) \|, \ldots , \| \rho^j(n) \|, \ldots) \\
     & \rightarrow \varrho = (\varrho_1, \ldots, \varrho_j, \ldots).
  \end{split}
\end{equation}
We write $\varrho_+$ for the limiting $\ell_1$ norm:
\begin{displaymath}
  \varrho_+ = \sum_j \varrho_j.
\end{displaymath}

\textit{Remark on Notation.} \ The index $j$, which runs over
principal components, will be written as a superscript on \textit{vectors}
$v^j, \rho^j$ and $u^j$(defined in Appendix), but as a subscript on
\textit{scalars} such as $\varrho_j(n)$ and $\varrho_j$.

\medskip

We continue to focus on the estimation of the principal eigenvector
$\rho^1$, and establish a more general version of the two preceding theorems.

\begin{theorem}
  \label{th:multi-cpt}
 Assume model (\ref{eq:mcm}) together with conditions (a)-(d). If $p/n
 \rightarrow c$, then
 \begin{displaymath}
   \limsup_{n \rightarrow \infty} \sin \angle ( \hat \rho^1, \rho^1) 
   \leq \frac{4 \sigma \sqrt c}{\varrho_1^2 - \varrho_2^2}[ \rho_+ +
   (2 + \sqrt c) \sigma]
 \end{displaymath}
so long as the right side is at most, say, 4/5. 

If $c > 0$, then
\begin{displaymath}
  \liminf_{n \rightarrow \infty} E \angle ( \hat \rho^1, \rho^1) > 0.
\end{displaymath}
\end{theorem}

Thus, it continues to be true in the multicomponent model
that $\hat \rho^1$ is consistent if and
only if $p = o(n)$.



\section{The sparse PCA algorithm}
\label{sec:sparse-pca-algorithm}

The inconsistency results of Theorems \ref{th:lowerbd} and
\ref{th:multi-cpt} emphasize the importance of reducing the number of
variables before embarking on PCA, and motivate the sparse PCA
algorithm to be described in general terms here. Note that the
algorithm \textit{per se} does not require the specification of a
particular model, such as (\ref{eq:mcm}).

1. \textit{Select Basis.} Select a basis $\{ e_\nu \}$ for
   $\mathbb{R}^p$ and compute co-ordinates $(x_{i \nu})$ for each $x_i$ in
   this basis:
   \begin{displaymath}
     x_i(t) = \sum_\nu x_{i \nu} e_\nu (t), \qquad i = 1, \ldots, n.
   \end{displaymath}

[The wavelet basis is used in this paper, for reasons discussed in the
next subsection.]

2. \textit{Subset.} Calculate the sample variances $\hat \sigma_\nu^2
   = \widehat{Var} (x_{i \nu})$. 
Let $\hat I$ denote the set of indices $\nu$ corresponding to the largest
   $k$ variances.

[$k$ may be specified in advance, or chosen based on the data, see 
Section \ref{sec:adaptive-choice-k} below].

3. \textit{Reduced PCA.} Apply standard PCA to the reduced data set
   $\{ x_{i \nu}, \nu \in \hat I, i = 1, \ldots, n\}$ on the
   selected $k-$dimensional subset, obtaining eigenvectors $\hat
   \rho^j = ( \hat \rho^j_\nu), j = 1, \ldots, k$.

4. \textit{Thresholding.} Filter out noise in the estimated
   eigenvectors by hard thresholding
   \begin{displaymath}
     \hat \rho^{*j}_\nu = \eta_H( \hat \rho^j_\nu, \delta).
   \end{displaymath}
[Hard thresholding is given, as usual, by $\eta_H(x,\delta) = x I \{
|x| \geq \delta \}$. An alternative is soft thresholding
$\eta_S(x,\delta) = \text{sgn}(x) (|x|-\delta)_+$, but hard
thresholding has been used here because it preserves the magnitude of
retained signals.

The threshold $\delta$ can be chosen, for example, by trial and error,
or as $\delta = \hat \tau_j \sqrt{2 \log k}$ for some estimate $\hat
\tau_j.$ In this paper, estimate (\ref{eq:taui}) is used. Another
possibility is to set $\hat \tau_j = MAD \{ \hat \rho^j_\nu,
\nu = 1, \ldots, k \}/ 0.6745$. ]

5. \textit{Reconstruction.} Return to the original signal domain,
   setting 
   \begin{displaymath}
     \hat \rho_j (t) = \sum_\nu \hat \rho^{*j}_\nu e_\nu (t).
   \end{displaymath}

In the rest of this section, we amplify on and illustrate various
aspects of this algorithm. Given appropriate eigenvalue and
eigenvector routines, it is not difficult to code. For example,
\texttt{MATLAB}
files that produce most figures in this paper will soon be available at 
\url{www-stat.stanford.edu/~imj/} -- to exploit wavelet bases, they make use of
the open-source library \texttt{WaveLab} available at
\url{www-stat.stanford.edu/~wavelab/}.

\subsection{Sparsity and Choice of basis}
\label{sec:spars-choice-basis}

Suppose that in the basis $\{ e_\nu (t) \}$  a population
principal component $\rho(t)$ has coefficients $\{ \rho_\nu \}$:
\begin{displaymath}
  \rho(t) = \sum_{\nu=1}^p \rho_\nu e_\nu(t).
\end{displaymath}

It is desirable, both from the point of view of economy of
representation, as well as computational complexity, for the expansion
in basis $\{ e_\nu \}$ to be \textit{sparse}, i.e., most coefficients
$\rho_\nu$ are small or zero.

One way to formalize this is to require that the ordered coefficient
magnitudes decay at some algebraic rate. We say that $\rho$ is
contained in a weak $\ell_q$ ball of radius $C$, $\rho \in w
\ell_q(C),$ if $|\rho|_{(1)} \geq |\rho|_{(2)} \geq \ldots $ and 
\begin{displaymath}
  |\rho|_{(\nu)} \leq C \nu^{-1/q}, \qquad \nu = 1, 2, \ldots
\end{displaymath}

Wavelet bases typically provide sparse representations of
one-dimensional functions that are smooth or have isolated
singularities or transient features, such as in our ECG example. Here
is one such result. Expand $\rho$ in a nice wavelet basis $\{
\psi_{jk}(t) \}$ to obtain $\rho = \sum_{jk} \rho_{jk} \psi_{jk}(t)$
and then order coefficients by absolute magnitude, so that
$(\rho_\nu)$ is a re-ordering of the $|\rho_{jk}|$ in decreasing
order.
Then smoothness (as measured by membership in some Besov space
$B^\alpha_{p,q}$) implies sparsity in the sense that 
\begin{displaymath}
      \rho \in B_{p,q}^\alpha \quad \Rightarrow \quad 
      (\rho_\nu) \in w \ell_p, \qquad p = 2/(2 \alpha + 1).
\end{displaymath}
[for details, see  \cite{dono93} and \cite{john01}: in particular it
is assumed that $\alpha > (1/p - 1/2)_+$ and that the wavelet $\psi$
is sufficiently smooth.]

In this paper, we will assume that the basis $\{ e_\nu \}$ is fixed in
advance -- and it will generally be taken to be a wavelet basis. 
Extension of our results to incorporate basis selection
(e.g. from a library of orthonormal bases such as wavelet packets) is
a natural topic for further research.

\subsection{Adaptive choice of $k$}
\label{sec:adaptive-choice-k}

Here are two possibilities for adaptive choice of $\hat k = | \hat I
|$ from the data:

(a) choose co-ordinates with variance exceeding the estimated noise
level by a specified fraction $\alpha_n$:
\begin{displaymath}
  \hat I = \{ \nu ~:~ \hat \sigma_\nu^2 \geq \hat \sigma^2 (1 + \alpha_n) \}.
\end{displaymath}
This choice is considered further in Section \ref{sec:corr-select-prop}.

(b) As motivation, recall that we hope that the selected set 
of variables $\hat I$ is
both small in cardinality and also captures most of the variance of
the population principal components, in the sense that the ratio
\begin{displaymath}
  \sum_{\nu \in \hat I} \rho_\nu^2 \Big/ \sum_\nu \rho_\nu^2
\end{displaymath}
is close to one for the leading population principal components in $\{
\rho^1, \ldots, \rho^m \}$.
Now let $\chi_{(n),\alpha}^2$ denote the upper $\alpha-$percentile of the
  $\chi_{(n)}^2$ distribution -- if all co-ordinates were pure noise,
  one might expect $\hat \sigma_{(\nu)}^2$ to be close to $n^{-1} \hat
  \sigma^2 \chi_{(n), \nu/n}^2$. Define the excess over these
  percentiles by 
  \begin{displaymath}
    \hat \tau_{(\nu)}^2 = \max \{ \hat \sigma_{(\nu)}^2 - n^{-1} \hat
  \sigma^2 \chi_{(n), \nu/n}^2, 0 \},
  \end{displaymath}
and for a specified fraction $w(n)$, set
\begin{displaymath}
  \hat I = \{ \nu ~:~ \sum_{\nu=1}^{\hat k} \hat \tau_{(\nu)}^2 \geq
  w(n) \sum_{\nu} \hat \tau_{(\nu)}^2 \},
\end{displaymath}
where $\hat k$ is the smallest index $k$ for which the inequality holds.
This second method has been used for the figures in this paper,
typically with $w(n) = .995$.

\medskip

\textit{Estimation of $\sigma$.}
If the population principal components $\rho^j$ have a sparse
representation in basis $\{ e_\nu \}$, then we may expect that in most
co-ordinates $\nu$, $\{ x_{i\nu} \}$ will consist largely of noise.
This suggests a simple estimate of the noise level
on the assumption that the noise level is the same in all co-ordinates, namely
\begin{equation}
  \label{eq:sigma2}
  \hat \sigma^2 = \text{median} (\hat \sigma_\nu^2 ).
\end{equation}



\subsection{Computational complexity}
\label{sec:comp-complexity}

It is straightforward to estimate the cost of sparse PCA by examining
its main steps:







\begin{enumerate}
\item  This depends on the choice of basis. In the wavelet case 
no more than $O(n  p  \log p)$ operations are needed. 
\item  Sort the sample variances and select $\hat I$: $O (p \log p)$.
\item  Eigendecomposition for a $k \times k$ matrix: $O(k^3)$.
\item  Estimate $\hat{\sigma}^2$ and $\widehat{\|\rho\|^2}$: $O(p)$.
\item  Apply thresholding: $O(k)$.
\item  Reconstruct eigenvectors in the original sample space:
$O(k^2 p)$.
\end{enumerate}

\noindent Hence, the total cost of sparse PCA is 
$$O( n  p \log p + k^2 p).$$

Both standard and smoothed PCA need at least $O((p\wedge n)^3)$
operations. Therefore, if we can find a sparse basis such that $k/p
\rightarrow 0$, then under the assumption that $p/n \rightarrow c$ as
$n \rightarrow \infty$,the total cost of sparse PCA is $o(p^3)$.
We will see in examples to follow that the savings can be substantial.



\subsection{Simulated examples}
\label{sec:simulated-examples}

The two examples in this section are both
motivated by functional data with localized features.  

\begin{figure*}[htb]
     \begin{center}
       \leavevmode
        \centerline{\includegraphics[ width = .6\textwidth,angle =
    0]{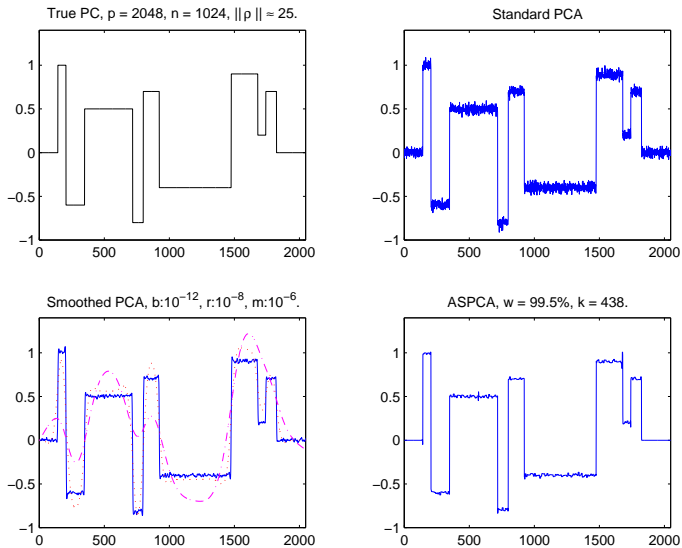}}
\caption{ \small Comparison of the sample principal components for a step
  function. \ \ \ \ \ (a)~True principal component $\rho_l = f(l/n)$, 
  the ``step'' function
  (b): Sample principal component by standard PCA. (c): Sample
  principal component by smoothed PCA using $\lambda = 10^{-12},
  10^{-8}$ and $10^{-6}$. (d): Sample principal component by 
  sparse PCA with weighting function $w = 99.5\%$, $k = 438$.  }
       \label{fig:step}
     \end{center}
\end{figure*}

The first is a three-peak principal component depicted in
Figure \ref{fig:betacurves}, and already discussed in Section
\ref{sec:introduction}. 
The second example, Figure \ref{fig:step}, has an underlying first principal
component composed of step functions.
For both examples, the dimension of data vectors is $p = 2048$, the
number of observations $n = 1024$, and the noise level $\sigma =1$.
However, the amplitudes of $\rho$  differ, with $\| \rho \| =10$
for the ``3-peak'' function and $\| \rho \| \approx 25$ for the ``step''
function.

Panels (d) and (b) in the two figures respectively
show the sample principal components
obtained by using standard PCA. While standard PCA does capture the
peaks and steps, it retains significant noise in the flat regions of
the function. Corresponding panels (e) and (c) 
show results from smooth PCA with the indicated
values of the smoothing parameter. Just as for the three peak curve
discussed earlier, in the case of the step function, none of the three
estimates simultaneously captures both jumps and flat regions well.



Panels (f) and (d) present the principal components obtained by sparse PCA.
Using method (b) of the previous section with $w = 99.5\%$,
the \textit{Subset} step selects $k = 372$ and 438
for the ``3-peak'' curve and ``step'' function, respectively. 
The sample principal component in Figure \ref{fig:betacurves}(d) is
clearly superior to the other sample p.c.s in Figure \ref{fig:betacurves}.
Although the principal component function in the step case appears to be
only slightly better than the solid blue smooth PCA estimate, we will
see later that its squared error is reduced by more than 90\%.

\bigskip

Table \ref{tab:tbd4} compares the accuracy of the three PCA algorithms, using
average squared error (ASE) defined as  
\begin{equation*}
{\rm ASE} = p^{-1} \| \hat \rho - \rho \|^2.
\end{equation*}
The average ASE over 50 iterations is shown.
The running time is the CPU time for a single iteration used by Matlab
on a MIPS R10000 195.0MHz server. 

Figure \ref{fig:aseplots} presents box plots of ASE for the 50
iterations. Sparse PCA gives the best result for the
``step'' curve. For the ``3-peak'' function, in only a few iterations
does sparse PCA generate larger error than smoothed PCA with a small
$\lambda = 10^{-12}$. On the average, ASE using sparse PCA is superior to
the other methods by a large margin.  Overall Table \ref{tab:tbd4} and
Figure \ref{fig:tbd3} show that sparse PCA leads to the most accurate
principal component while using much less CPU time than other PCA
algorithms.

\begin{center}
\begin{table*}
\begin{center}{\begin{tabular}{|c|c|c|c|c|}\hline
 & Standard & Smoothed & Smoothed & Sparse \\
 & PCA & $\lambda:10^{-12}$ &  $\lambda:10^{-6}$ & PCA \\ 
\hline
ASE (3-peak) & 9.681e-04 &  1.327e-04  &  3.627e-2 &  7.500e-05 \\
\hline
Time (3-peak) & $\sim$ 12min & $\sim$ 47 min & $\sim$ 43 min & 1 min 15 s \\
\hline
ASE (step) & 9.715e-04 &  3.174e-3  &  1.694e-2 &  1.947e-04 \\
\hline
Time (step) & $\sim$ 12min & $\sim$ 47 min & $\sim$ 46 min & 1 min 31 s \\
\hline
\end{tabular} \\
 \caption{Accuracy and efficiency comparison}\label{tab:tbd4}}
\end{center}
\end{table*}
\end{center}

\begin{figure*}[htb]
     \begin{center}
       \leavevmode
\centerline{\includegraphics[ width = .4\textwidth,angle =
    -90]{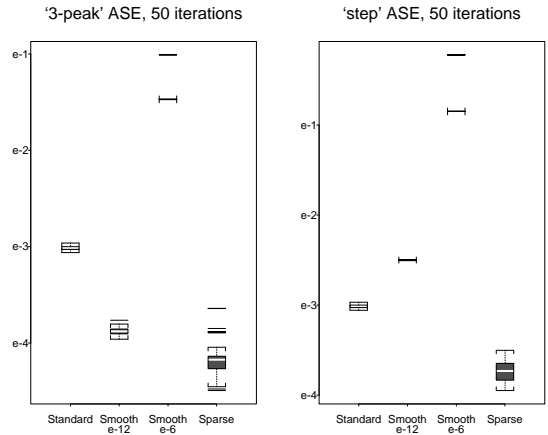}}
\caption{ \small Side-by side box-plots of ASE from 50 iterations using 
different algorithms. (a) For the ``3-peak'' function. (b) For the ``step''
  function.}
       \label{fig:aseplots}
     \end{center}
\end{figure*}

\bigskip \noindent
\textit{Remarks on the single component model.}

Anderson (1963) obtained the asymptotic distribution of 
$\sqrt{n}(\rho -\hat{\rho})$ for \textit{fixed} $p$;
in particular
\begin{equation}
{\rm Var} \{\sqrt{n}(\rho_\nu - \hat{\rho_\nu})\} 
\rightarrow (\|\rho\|^2 + \sigma^2)
\frac{\sigma^2}{\|\rho\|^4}(1-\rho_\nu^2),
\label{eq:vardiff}
\end{equation}
as $n \rightarrow \infty.$ 
For us, $p$ increases with $n$, but we will nevertheless use
(\ref{eq:vardiff}) as an heuristic basis for estimating the variance
$\hat \tau$ needed for thresholding.
Since the effect of thresholding is to remove noise in small
coefficients, setting $\rho_\nu$ to 0 in $(\ref{eq:vardiff})$ suggests
\begin{equation}
\hat{\tau}_\nu \approx \frac{1}{\sqrt{n}} \frac{\sigma
  \sqrt{\|\rho\|^2 + \sigma^2}}{\|\rho\|^2}.
\label{eq:taui}
\end{equation}

Neither $\|\rho\|^2$ and $\sigma^2$ in $(\ref{eq:taui})$ are known,
but they can be estimated by using the information contained in 
the sample covariance matrix $S$, much as in the discussion of
Section \ref{sec:adaptive-choice-k}. 
Indeed $S_\nu^2$, the $\nu$-th diagonal element of $S$, 
follows a scaled $\chi^2$
distribution, with expectation $\rho_\nu^2 + \sigma^2.$
If $\rho_\nu$ is a sparse representation of $\rho$, then most
coefficients will be small, suggesting the estimate (\ref{eq:sigma2}) for
$\sigma^2$.
In the single component model,
\begin{equation*}
||\rho||^2 = \sum_1^p \rho_\nu^2 = \sum_1^p  E(S_\nu^2) 
- \sigma^2,
\label{eq:rhonorm2}
\end{equation*}
which suggests as an estimate:
\begin{equation}
\widehat{||\rho||^2} = 
\sum_1^p \bigl\{ S_\nu^2 - {\rm median}(S_\nu^2) \bigr\}.
\label{eq:rhonormest}
\end{equation} 
Figure \ref{fig:tbd3} shows the histograms for these estimates of
$\|\rho\|$ and $\sigma$
based on 100 iterations for the ``3-peak'' curve and for the ``step''
function.

\begin{figure*}[htb]
     \begin{center}
       \leavevmode
       \centerline{\includegraphics[ width = .6\textwidth,angle =
          0]{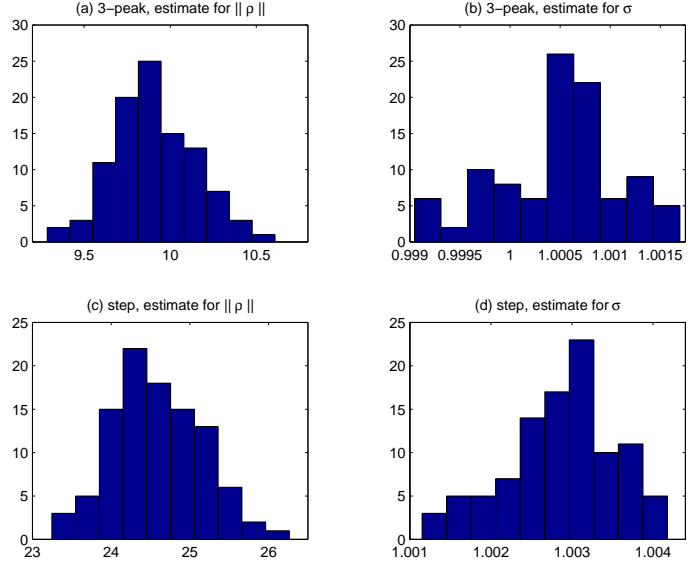}}
\caption{ \small Histograms from 100 iterations.  The ``3-peak''
  function, (a) estimate for $\|\rho\| =10$: mean = 9.91, SD = 0.24.  
(b): estimate for $\sigma = 1$: mean = 1.0005, SD = .0006.
The ``step''  function,  (c): estimate for $\|\rho\| = 24.82$:
mean = 24.58, SD = 0.56.
(d): estimate for $\sigma =1$: mean = 1.0029, SD = .0007.}
       \label{fig:tbd3}
     \end{center}
\end{figure*}

\subsection{Correct Selection Properties}
\label{sec:corr-select-prop}

A basic issue raised by the sparse PCA algorithm is whether the
selected subset $\hat I$ in fact correctly contains the largest
population variances, and only those. We formulate a result, based on
large deviations of $\chi^2$ variables, that provides some reassurance.

For this section, assume that the diagonal elements of the sample
covariance matrix $S = n^{-1} \sum_1^n x_i x_i^T$ have marginal
$\chi^2$ distributions, i.e.,
\begin{equation}
  \label{eq:chisqassn}
  \hat \sigma_\nu^2 = S_{\nu \nu} \sim \sigma_\nu^2 \chi_{(n)}^2 /n,
  \qquad \nu = 1, \ldots, p.
\end{equation}
We will not require any assumptions on the joint distribution of $\{
\hat \sigma_\nu^2 \}$.

Denote the ordered population coordinate variances by $\sigma_{(1)}^2
\geq \sigma_{(2)}^2 \geq \ldots$ and the ordered sample coordinate
variances by 
$\hat \sigma_{(1)}^2 \geq \hat \sigma_{(2)}^2 \geq \ldots $. 
A desirable property is that $\hat I$ should, for suitable
$\alpha_n$ small,

\noindent (i) include \textit{all} indices $l$ in 
\begin{align*}
  I_{in} & = \{ l: \sigma_l^2 \geq \sigma_{(k)}^2 (1 + \alpha_n) \}, 
\qquad \text{and} \\
\intertext{(ii) exclude \textit{all} indices $l$ in }
  I_{out} & = \{ l: \sigma_l^2 \leq \sigma_{(k)}^2 (1 - \alpha_n) \}.
\end{align*}
We will show that this in fact occurs if $\alpha_n = \gamma \sqrt{
n^{-1} \log n}$, for appropriate $\gamma > 0.$

We say that a \textit{false exclusion} (FE) occurs if any variable in $I_{in}$
is missed:
\begin{displaymath}
  FE = \bigcup_{l \in I_{in}} \{ \hat \sigma_l^2 < \hat \sigma_{(k)}^2 \},
\end{displaymath}
while a \textit{false inclusion} (FI) happens if any variable in $I_{out}$ is
spuriously selected:
\begin{displaymath}
  FI = \bigcup_{l \in I_{out}} \{ \hat \sigma_l^2 \geq \hat \sigma_{(k)}^2 \}.
\end{displaymath}

\begin{theorem}
  \label{th:fe-fi}
Under assumptions (\ref{eq:chisqassn}), the chance of an inclusion
error of either type in $\hat I_k$ having magnitude $\alpha_n = \gamma
n^{-1/2} (\log n)^{1/2}$ is polynomially small:
\begin{displaymath}
  P \{ FE \cup FI \} \leq 2 p k n^{-b(\gamma)} 
            + k n^{-(1- 2 \alpha_n) b(\gamma)},
\end{displaymath}
with $b(\gamma) = [ \gamma \sqrt 3/( 4 + 2 \sqrt 3)]^2.$
\end{theorem}

For example, if $\gamma = 9$, then $b(\gamma) \doteq 4.36.$
As a numerical illustration based on (\ref{eq:p-bound}) below, 
if the subset size $k = 50$, while $p = n
= 1000$, then the chance of an inclusion error corresponding to a 25\%
difference in SDs (i.e. $\sqrt{1 + \alpha_n} = 1.25$) is below 5\%.
That reasonably large sample sizes are needed is a sad fact inherent
to variance estimation---as one of Tukey's `anti-hubrisines' puts it,
``it takes 300 observations to estimate a variance to one significant
digit of accuracy''. 

\subsection{Consistency}
\label{sec:consistency}

The sparse PCA algorithm is motivated by the idea that if the p.c.'s
have a sparse representation in basis $\{ e_\nu \}$, then selection of
an appropriate subset of variables should overcome the inconsistency
problem described by Theorem \ref{th:lowerbd}.

To show that such a hope is justified, we establish a consistency
result for sparse PCA. For simplicity, we consider the single
component model (\ref{eq:singlecpt}), and assume that $\sigma^2$ is
known---though this latter assumption could be removed by estimating
$\sigma^2 $ using (\ref{eq:sigma2}). 

To select the subset of variables $\hat I$, we use a version of rule
(a) from Section \ref{sec:adaptive-choice-k}:
\begin{equation}
  \label{eq:selectrule}
  \hat I = \{ \nu ~:~ \hat \sigma_\nu^2 \geq \sigma^2(1+\gamma_n) \},
\end{equation}
with $\gamma_n = \gamma (n^{-1} \log n)^{1/2}$ and $\gamma$ a
sufficiently large positive constant---for example $\gamma > \sqrt{12}$ 
would work for the proof.

We assme that the unknown principal components $\rho = \rho(n)$
satisfy a \textit{uniform sparsity condition}: for some positive
constants $q, C$,
\begin{equation}
  \label{eq:unif-sparsity}
  \rho(n) \in w \ell_q(C) \quad \text{uniformly in} \ n.
\end{equation}
Let $\hat \rho_I$ denote the principal eigenvector estimated by step (3)
of the sparse PCA algorithm (thresholding is not considered here).

\begin{theorem}
  \label{th:consistency}
Assume that the single component model (\ref{eq:singlecpt}) holds,
with $p/n \rightarrow c > 0$ and $\| \rho(n) \| \rightarrow \varrho >
0$. For each $n$, assume that $\rho(n)$ satisfies the uniform sparsity
condition (\ref{eq:unif-sparsity}).

Then the estimated principal eigenvector $\hat \rho_I$ obtained by subset
selection rule (\ref{eq:selectrule}) is consistent:
\begin{displaymath}
  \angle( \hat \rho_I, \rho) \stackrel{a.s.}{\rightarrow} 0.
\end{displaymath}
\end{theorem}

The proof is given in Appendix \ref{sec:proof-theorem-consist}: it is
based on a correct selection property similar to Theorem
\ref{th:fe-fi}: combined with a modification of the consistency
argument for Theorem \ref{th:multi-cpt}. 
In fact, the proof shows that consistency holds even under the weaker
assumption $p = O(n^a)$, for arbitrary $a > 0$, so long as $\gamma =
\gamma(a)$ is set sufficiently large.


\subsection{ECG example}
\label{sec:ecg-example}

This section offers a brief illustration of sparse PCA as applied to
some ECG data kindly provided by Jeffrey Froning and Victor Froelicher
in the cardiology group at Palo Alto Veterans Affairs Hospital.  Beat
sequences -- typically about 60 cycles in length -- were obtained from
some 15 normal patients: we have selected two for the preliminary
illustrations here.


\medskip
\textit{Data Preprocessing.} \ 
Considerable preprocessing is routinely done on ECG signals before the
beat averages are produced for physician use. Here we describe certain
steps taken with our data, in collaboration with Jeff Froning,
preparatory to the PCA analysis.



The most important feature of an ECG signal is
the Q-R-S complex: the maximum occurs at the
R-wave, as depicted in Figure \ref{fig:ecgsamp}(b).
Therefore we define the length of one cycle as the gap between two
adjacent maxima of R-waves. 


\begin{figure*}[htb]
     \begin{center}
       \leavevmode
       \centerline{\includegraphics[ width = \textwidth,angle =
         0]{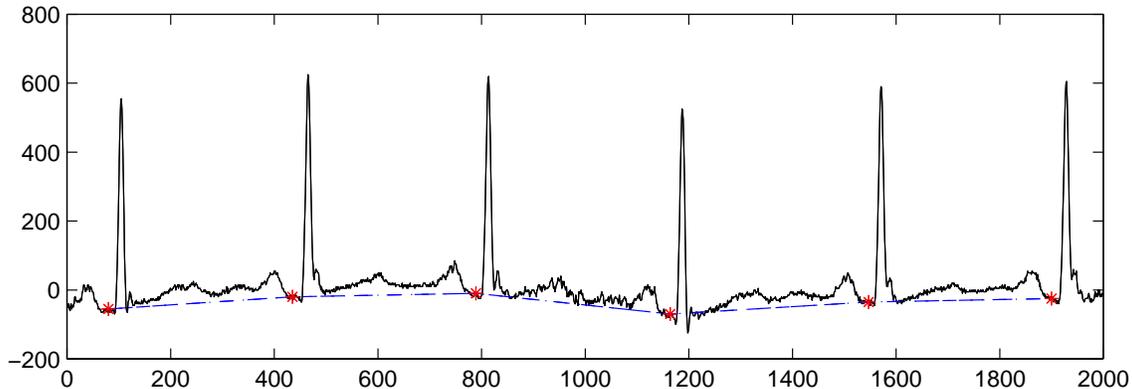}}
\caption{ \small ECG baseline wander.}
       \label{fig:tbd4}
     \end{center}
\end{figure*}

1. \textit{Baseline wander} is observed in many ECG data sets, c.f.
Figure \ref{fig:tbd4}. One common remedy for this problem is to deduct
a piecewise linear baseline from the signal, the linear segment
(dashed line) between two beats being determined from two adjacent
onset points.

The onset positions of R-waves are shown by asterisks.  Their
exact locations vary for different patients, and as Figure
\ref{fig:tbd4} shows, even for adjacent R-waves. The locations are
determined manually in this example.  To reduce the effect of noise,
the values of onset points are calculated by an average of 5 points
close to the onset position.

2. Since pulse rates vary even on short time scales,
the duration of each heart beat cycle may vary as well. We
use linear interpolation to equalize the duration of each cycle, and
for convenience in using wavelet software, 
discretize to $512 = 2^9$ sample points in each cycle.  

3. Finally, due to the
importance of the R-wave, the horizontal positions of the maxima are
the 150th position in each cycle.

4. Convert the ECG data vector into an $n \times p$ data matrix, where
$n$ is the number of observed cycles and $p=512$.
Each row of the matrix presents one heart 
beat cycle with the maxima of R-waves all aligned at the same
position.



\medskip




\begin{figure*}[htb]
     \begin{center}
       \leavevmode
       \centerline{\includegraphics[ width = .8\textwidth,angle =
        0]{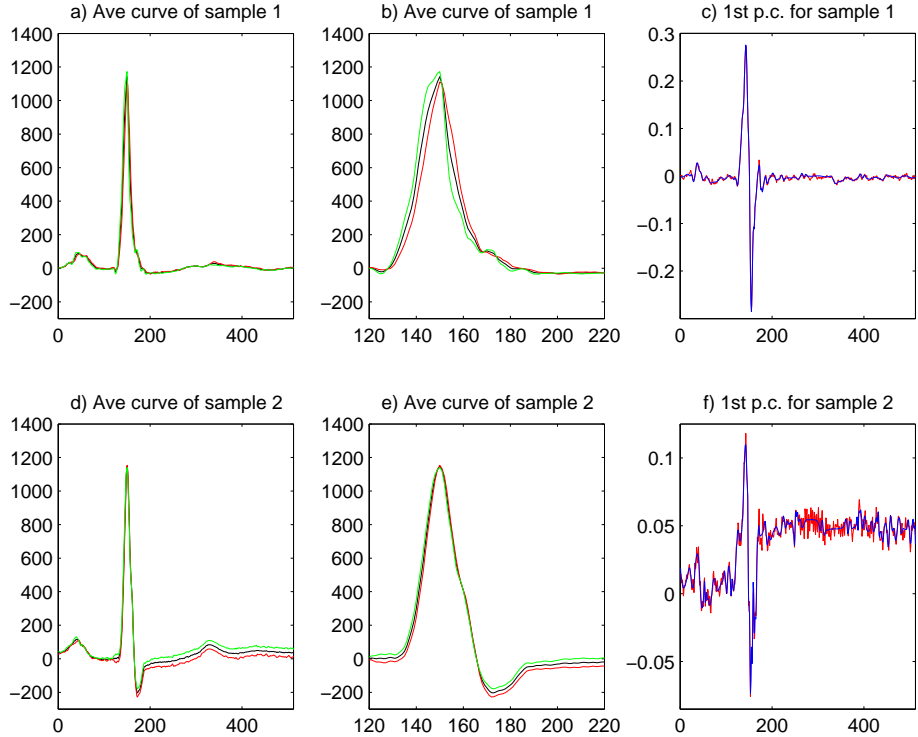}}
\caption{ \small ECG examples. (a): mean curve for ECG sample 1, $n = 66$,
 in blue, along with $\bar x + 2 \hat \rho$ (green) and $\bar x - 2
\hat \rho$ (red), with $\hat \rho$ being the estimated first principal
component from sparse PCA (see also (c)). 
(b) Magnified section of (a) over the range 120-220.
(c): First principal components for sample 1 from standard (red) and
sparse PCA (blue). 
(d)-- (f): corresponding plots for sample 2, $n = 61$.}
       \label{fig:tbd5}
     \end{center}
\end{figure*}

\textit{PCA analysis.}
Figure \ref{fig:tbd5} (a) and (d) shows the mean curves for two ECG
samples in blue. The number of observations $n$, i.e. number of heart beats
recorded, are 66 and 61, respectively. The first sample principal
components for these two sample sets are plotted in plots (c) and
(f), with
red curves from standard PCA and blue curves from sparse PCA. 
In both cases
there are two sharp peaks in the vicinity of the QRS complex. The
first peak occurs shortly before the 150th position, where all the maxima 
of R-waves are aligned,
and the second peak, which has an opposite sign, shortly after.


The standard PCA curve in Figure 6.7.(b, red) is less noisy
than that in panel (d, red), even allowing for the difference
in vertical scales. 
Using $(\ref{eq:sigma2})$, 
$$\hat{\sigma}_1^2 = 24.97 \quad {\rm and} \quad \hat{\sigma}_2^2 = 82.12.$$
while the magnitudes of the
two mean sample curves are very similar. 


The sparse PCA curves (blue) are smoother than the standard PCA ones
(red), especially in plot (d) where the signal to noise ratio is
lower. On the other hand, the red and blue curves match quite well at
the two main peaks. Sparse PCA has reduced noise in the sample
principal component in the baseline while keeping the main features.

There is a notable difference between the two estimated p.c.'s.  In
the first case, the p.c. is concentrated around the R-wave maximum,
and the effect is to accelerate or decelerate the rise (and fall) of
this peak from baseline in a given cycle.  This is more easily seen by
comparing plots of $\bar x + 2 \hat \rho$ (green) with $\bar x - 2
\hat \rho$ (red), shown over a magnified part of the cycle in panel
(b).  In the second case, the bulk of the energy of the p.c. is
concentrated in a level shift in the part of the cycle starting with
the ST segment.  This can be interpreted as beat to beat fluctuation
in baseline -- since each beat is anchored at $0$ at the onset point,
there is less fluctuation on the left side of the peak. This is
particularly evident in panel (e) -- there is again a slight
acceleration/deceleration in the rise to the R wave peak  -- less
pronounced in the first case, and also less evident in the fall.

Obvious questions raised by this illustrative example include the
nature of effects which may have been introduced by the preprocessing
steps, notably the baseline removal anchored at onset points and the
alignment of R-wave maxima. 
Clearly some conventions must be adopted to create rectangular data
matrices for p.c. analysis, but detailed analysis of these issues must
await future work.

Finally, sparse PCA uses less than 
10\% of the computing time than standard PCA.


\appendix
\section{Appendix}
\label{sec:appendix}


\subsection{Preliminaries}
\label{sec:preliminaries}

\textbf{Matrices.}
We first recall some pertinent matrix results. Define the $2-$norm of
a rectangular matrix by
\begin{equation}
  \label{eq:normdef}
\| A \|_2 = \sup \{ \| Ax \|_2 : \| x \|_2 = 1 \}.
\end{equation}
If $A$ is real and symmetric, then $\| A \|_2 = \lambda_{max}(A).$
If $A_{p \times p}$ is partitioned
\begin{displaymath}
  A = 
  \begin{pmatrix}
    a & b^T \\ b & C
  \end{pmatrix}
\end{displaymath}
where $b$ is $(p-1) \times 1$, then by setting $x = ( 1 \  0^T)^T$
in (\ref{eq:normdef}), one finds that
\begin{equation}
  \label{eq:elem}
  \| b \|_2 \leq \| A \|_2.
\end{equation}

The matrix $B = \rho u^T + u \rho^T$ has at most two non-zero eigenvalues,
given by
\begin{equation}
  \label{eq:ranktwo}
  \lambda = (\tau \pm 1) \| \rho \| \| u \|, \qquad 
  \tau = \rho^T u/ \| \rho \| \| u \|.
\end{equation}
Indeed, the identity $\det (I + AC) = \det (I + CA)$ for compatible
rectangular matrices $A$ and $C$ means that the
non-zero eigenvalues of 
\begin{displaymath}
  B = \begin{pmatrix} \rho & u  \end{pmatrix}
  \begin{pmatrix} u^T \\ \rho^T  \end{pmatrix}
\end{displaymath}
are the same as those of the $2 \times 2$ matrix
\begin{displaymath}
  B^* = \begin{pmatrix} u^T \\ \rho^T  \end{pmatrix}
         \begin{pmatrix} \rho & u  \end{pmatrix} 
      =  \begin{pmatrix}
        \tau \| \rho \| \| u \| &  \| u \|^2 \\
        \| \rho \|^2  &  \tau \| \rho \| \| u \|
      \end{pmatrix}
\end{displaymath}
from which (\ref{eq:ranktwo}) is immediate.

\bigskip

\textbf{Angles between vectors.} \ 
We recall and develop some elementary facts about angles between
vectors.
The angle between two non-zero vectors $\xi, \eta$ in $\mathbb{R}^p$
is defined as
\begin{equation}
  \label{eq:angle}
  \angle (\xi,\eta) = \frac{\cos^{-1} |\xi^T \eta |}{\| \xi \|_2 \|
  \eta \|_2} \in [0, \pi/2].
\end{equation}
Clearly $\angle (a \xi, b \eta) = \angle (\xi,\eta)$ 
for non-zero scalars $a$ and $b$; 
in fact $\angle ( \cdot, \cdot)$ is a metric on one-dimensional
subspaces of $\mathbb{R}^p$. 
If $\xi$ and $\eta$ are chosen to be unit vectors with 
$\xi^T \eta \geq 0$, then
\begin{equation}
  \label{eq:halfangle}
  \| \xi - \eta \|_2 = 2 \sin \hf \angle (\xi, \eta).
\end{equation}

The sine rule for plane triangles says that if $\xi, \eta$ are
non-zero and linearly independent vectors in $\mathbb{R}^p$, then
\begin{equation}
  \label{eq:sinerule}
  \sin \angle (\xi,\eta) = \frac{\| \xi - \eta \|}{\| \xi \|} \sin
  \angle (\xi - \eta ,\eta).  
\end{equation}

These remarks can be used to bound the angle between a vector $\eta$
and its image under a symmetric matrix $M$ in terms of the angle
between $\eta$ and any principal eigenvector of $M$.

\begin{lemma}
\label{lem:perturb}
  Let $\xi$ be a principal eigenvector of a non-zero symmetric matrix
  $M$. For any $\eta \neq 0$,
  \begin{displaymath}
    \angle ( \eta, M \eta) \leq 3 \angle (\eta, \xi). 
  \end{displaymath}
\end{lemma}
\begin{proof}
  We may assume without loss of generality that $\| \xi \| =  \| \eta
  \| = 1$ and
  that $\xi^T \eta \geq 0$. Since $\xi$ is a principal eigenvector of
  a symmetric matrix, $ \| M \xi \| = \| M \|$. From the sine rule
  (\ref{eq:sinerule}),
  \begin{align*}
    \sin \angle (M \xi, M\eta) 
    & \leq \| M \xi - M \eta \|/ \| M \xi \| \\
    & \leq \| \xi - \eta \| 
    = 2 \sin \hf \angle (\xi, \eta),
  \end{align*}
where the final equality uses (\ref{eq:halfangle}). 
Some calculus shows that $2 \sin \alpha/2 \leq \sin 2 \alpha$ for $0
\leq \alpha \leq \pi/4$ and hence
\begin{equation}
  \label{eq:bound}
  \angle( M \xi, M \eta) \leq 2 \angle( \xi, \eta).
\end{equation}
From the triangle inequality on angles,
\begin{align*}
  \angle( \eta, M \eta) & \leq \angle( \eta, M \xi) + \angle( M \xi, M
  \eta) \\
  & \leq 3 \angle( \eta, \xi),
\end{align*}
using (\ref{eq:bound}) and the fact that $\xi$ is an eigenvector of $M$.
\end{proof}

\bigskip

\textbf{Perturbation bounds.} \ 
Suppose that a symmetric matrix $A_{p \times p}$ has unit eigenvector
$q_1$.  We wish to bound the effect of a \textit{symmetric}
perturbation $E_{p \times p}$ on $q_1$.
The following result (\citet[Thm 8.1.10]{govl96}, see also \cite{stsu90})
constructs a unit eigenvector $\hat q_1$ of $A+E$ and bounds its
distance from $q_1$ in terms of $\| E \|_2$.
Here, the distance between unit eigenvectors $q_1$ and $\hat q_1$ is defined
as at (\ref{eq:distdef}) and (\ref{eq:angle}).

Let $Q_{p \times p} = [q_1 \  Q_2]$ be an orthogonal matrix 
containing $q_1$ in the first column, and partition conformally
\begin{displaymath}
  Q^T A Q = \begin{pmatrix} \lambda & 0 \\ 0 & D_{22} \end{pmatrix},
  \quad
  Q^T E Q = \begin{pmatrix} \epsilon & e^T \\ e & E_{22} \end{pmatrix},
\end{displaymath}
where $D_{22}$ and $E_{22}$ are both $(p-1) \times (p-1)$.

Suppose that $\lambda$ is separated from the rest of the spectrum of
$A$; set
\begin{displaymath}
  \delta = \min_{\mu \in \lambda(D_{22})} | \lambda - \mu|.
\end{displaymath}
If $\| E \|_2 \leq \delta/5$, then there exists $r \in
\mathbb{R}^{p-1}$ satisfying
\begin{equation}
  \label{eq:rbound}
  \| r \|_2 \leq (4/\delta) \| e \|_2
\end{equation}
such that
\begin{displaymath}
  \hat q_1 = ( 1 + r^T r)^{-1/2} (q_1 + Q_2 r)
\end{displaymath}
is a unit eigenvector of $A+E$.
Moreover,
\begin{displaymath}
  \text{dist} (\hat q_1, q_1 ) \leq (4/\delta) \| e \|_2.
\end{displaymath}
Let us remark that since $\| e \|_2 \leq \| E \|_2$ by (\ref{eq:elem}), we have
$\| r \|_2 \leq 1$ and
\begin{equation}
    \label{eq:q1bd}
  q_1^T \hat q_1 = ( 1 + \| r \|_2^2 )^{-1/2} \geq 1/\sqrt 2.
\end{equation}

\bigskip

Suppose now that $q_1$ is the eigenvector of $A$ associated
with the \textit{principal} eigenvalue $\lambda_1(A)$. We verify that, under
the preceding conditions, $\hat q_1$ is also the principal eigenvector 
of $A+E$: i.e. if $(A+E) \hat q_1 =
\lambda^* \hat q_1$, then in fact $\lambda^* = \lambda_1(A+E)$.

To show this, we verify that $\lambda^* > \lambda_2(A+E)$.
Take inner products with $q_1$ in the eigenequation
for $\hat q_1$:
\begin{equation}
  \label{eq:lamstareq}
  \lambda^* q_1^T \hat q_1 = q_1^T A \hat q_1 + q_1^T E \hat q_1.
\end{equation}
Since $A$ is symmetric, $q_1^T A = \lambda_1(A) q_1^T$.
Trivially, we have  $q_1^T E \hat q_1 \geq - \| E \|_2$.
Combine these remarks with (\ref{eq:q1bd}) to get
\begin{displaymath}
  \lambda^* \geq \lambda_1(A) - \sqrt 2 \| E \|_2.
\end{displaymath}

Now $\delta = \lambda_1(A) - \lambda_2(A)$ and since from the minimax 
characterization of eigenvalues (e.g. \citet[p. 396]{govl96} or
\citet[p.218]{stsu90}), $\lambda_2(A+E) \leq \lambda_2(A) + \| E \|_2$, we
have 
\begin{align*}
  \lambda^* - \lambda_2(A+E) 
  & \geq \delta - (1 + \sqrt 2) \| E \|_2 \\
  & \geq \delta [ 1 - (1+ \sqrt 2)/5] > 0,
\end{align*}
which is the inequality we seek.


\bigskip

\textbf{Large Deviation Inequalities.}
If $\bar X = n^{-1} \sum_1^n X_i$ is the average of i.i.d. variates
with moment generating function $\exp \{\Lambda(\lambda) \} = E \exp
\{ \lambda X_1 \}$, then Cramer's theorem (see e.g. \citet[2.2.2 and
2.2.12]{deze93}) says that for $x > E X_1$,
\begin{equation}
  \label{eq:deze}
  P\{ \bar X > x \} \leq \exp \{ -n \Lambda^*(x) \},
\end{equation}
where the conjugate function $\Lambda^*(x) = \sup_{\lambda } 
\{ \lambda x - \Lambda(\lambda) \}$. The same bound holds for 
$P \{ \bar X < x \}$ when $x < EX_1$.

When applied to the $\chi_{(n)}^2$ distribution, with $X_1 = z_1^2$
and $z_1 \sim N(0,1)$,
the m.g.f. $\Lambda(\lambda) = - \hf \log (1 - 2 \lambda)$ and the
conjugate function $\Lambda^*(x) = \hf [ x - 1 - \log x]$. The bounds
\begin{displaymath}
  \log( 1 + \epsilon) \leq 
  \begin{cases}
    \epsilon - \epsilon^2 / 2   &  \ \  -1 < \epsilon < 0, \\
    \epsilon - 3 \epsilon^2 / 8 &  \ \  0 \leq \epsilon < \hf,
  \end{cases}
\end{displaymath}
(the latter following, e.g., from (47) in \cite{john00}) yield
\begin{alignat}{2}
  P \{ \chi_{(n)}^2 \leq n(1-\epsilon) \} & \leq \exp \{ -n \epsilon^2
  / 4 \}, \quad && 0 \leq \epsilon < 1,  \label{eq:chi2lo} \\
  P \{ \chi_{(n)}^2 \geq n(1+\epsilon) \} & \leq \exp \{ -3n \epsilon^2
  / 16 \}, \quad && 0 \leq \epsilon < \hf  \label{eq:chi2hi}.
\end{alignat}

We will use also a slightly sharper bound 
\begin{equation}
  \label{eq:sharper}
  P \{ \chi^2_{(n)} \geq n + t \sqrt{2n} \} \leq t^{-1} e^{-t^2/2}.
\end{equation}
valid for $n \geq 16$ and $0 \leq t \leq n^{1/6}$ \citep{john00}.

When applied to sums of variables $X_1 = z_1 z_2$,
with $z_1$ and $ z_2$ independent $N(0,1)$ variates, the m.g.f. 
$\Lambda(\lambda) = - \hf \log ( 1 - \lambda^2)$. 
With $\lambda_*(x) = [(1+4x^2)^{1/2} - 1]/(2x),$ the conjugate
function satisfies
\begin{displaymath}
    \Lambda^*(x) = \lambda_* x + \hf \log (1 - \lambda_*^2) 
   = (3/2) x^2 + O(x^4),
\end{displaymath}
as $x \rightarrow 0$. Hence, for $n$ large,
\begin{equation}
  \label{eq:3b2bd}
  P \{ \bar X > \sqrt{ b n^{-1} \log n} \} \leq C n^{-3b/2}.
\end{equation}


\textbf{Decomposition of sample covariance matrix.}
Now adopt the multicomponent model (\ref{eq:mcm}) along with its
assumptions (a) - (c).
The sample covariance matrix $S = n^{-1} \sum_{1}^{n} x_i x_i^{T}$ has
expectation $ES = R + \sigma^2 I_p$, where
\begin{equation}
  \label{eq:Rdecomp}
  R = \sum_{j=1}^m \rho^j \rho^{jT}.
\end{equation}
Now decompose $S$ according to (\ref{eq:mcm}).
Introduce $1 \times n$ row vectors $v^{jT} = (v_1^j \cdots v_n^j)$ and
collect the noise vectors into a matrix $Z_{p \times n} = [ z_1 \cdots
z_n]$. We then have
\begin{equation}
S - E S = \sum_{j, k = 1}^m A^{jk} + \sum_{j=1}^m B^{j}+C.
\label{eq:sabc}
\end{equation}
where the $p \times p$ matrices
\begin{equation}
  \label{eq:abc}
  \begin{split}
    A^{jk} &= \Bigl( n^{-1} \sum_{i=1}^n
    v_i^{j}v_i^k - \delta_{jk} \Bigr) 
    \rho^{j}\rho^{kT} = v_s^{jk} \rho^{j}\rho^{kT}, \\
    B^{j} &= \sigma n^{-1} \left(\rho^{j}v^{jT}Z^{T} + Z
      v^{j}\rho^{jT}\right),\\
    C &= \sigma^2 \bigl( n^{-1} ZZ^T - I_p \bigr).
  \end{split}
\end{equation}

\bigskip
\textbf{Some limit theorems.}
We turn to properties of the noise matrix $Z$ appearing in (\ref{eq:abc}). 
The cross products
matrix $Z Z^T$ has a standard $p$-dimensional Wishart $W_p(n,I)$
distribution with $n$ degrees of freedom and identity covariance
matrix, see  e.g. \citet[p82]{muir82}.  Thus the matrix $C = (c_{jk})$ in
(\ref{eq:sabc}) is simply a scaled and recentered Wishart matrix.
We state results below in terms of either $Z Z^T$ or $C$, depending on
the subsequent application.
Properties (b) and (c) especially play a key role in inconsistency
when $c>0$.

\bigskip

(a) If $p = O(n)$, then for any $b > 8$,
\begin{equation}
  \label{eq:maxcge}
  \max_{j, k} |c_{jk}| \leq \sigma \sqrt{\frac{b \log n}{n}} \qquad
  a.s.\qquad \mbox{as } \ n \rightarrow \infty.
\end{equation}
\begin{proof}
We may clearly take $\sigma = 1.$ 
An off-diagonal term in $n^{-1} Z Z^T = (c_{jk})$ has the distribution of 
an i.i.d. average $\bar X = n^{-1} \sum X_i$ where $X_1 = z_1 z_2$ is
the product of two independent standard normal variates. Thus
\begin{equation}
  \label{eq:psquared}
  P \{ \max_{j \neq k} |c_{jk}| > x \} \leq 2 p^2 P \{ \bar X > x \}.
\end{equation}
Now apply the large deviation bound (\ref{eq:3b2bd}) to the right
hand side. Since $p \sim cn$, the Borel-Cantelli lemma suffices to
establish (\ref{eq:maxcge}) for off-diagonal elements
for any $b>2$.

A diagonal term $c_{jj} +1$ in $n^{-1} Z Z^T$ has the $n^{-1}
\chi^2_{(n)}$ distribution. 
Setting $t = \sqrt{ \hf b \log n}$ in (\ref{eq:sharper}) yields
\begin{displaymath}
  P \{ c_{jj} > \sqrt{ b n^{-1} \log n} \} \leq
  \sqrt{2} ( b \log n)^{-1/2} n^{-b/4}.
\end{displaymath}
Since there are $p \sim cn$ diagonal terms, the conclusion
(\ref{eq:maxcge}) follows (again via Borel-Cantelli) so long as $b > 8$.
\end{proof}

\bigskip

(b) \citet{gema80} and \cite{silv85} respectively
established almost sure limits for the largest and
smallest eigenvalues of a $W_p(n,I)$ matrix as $p/n \rightarrow c \in
[0,\infty)$, from which follows:
\begin{equation}
  \label{eq:geman}
  \lambda_1(C), \lambda_p(C) \rightarrow \sigma^2 (c \pm 2 \sqrt{c}).
\end{equation}
[Although the results in the papers cited are for $c \in (0,\infty)$,
the results are easily extended to $c=0$ by simple coupling arguments.]


\bigskip

(c) Suppose in addition that $v$ is a $1 \times n$ vector with
independent $N(0,1)$ entries, which are also independent of $Z$.
Conditioned on $v$, the vector $Zv$ is distributed as $N_p(0, \| v
\|^2 I).$ Since $Z$ is independent of $v$, we conclude that
\begin{equation}
  \label{eq:repn}
  Z v \stackrel{\mathcal{D}}{=} \chi_{(n)} \chi_{(p)} U_p
\end{equation}
where $\chi_{(n)}^2$ and $\chi_{(p)}^2$ denote chi-square variables
and $U_p$ a vector uniform on the surface of the unit sphere $S^{p-1}$
in $\mathbb{R}^p$, and all three variables are independent.

Now let $u_{p \times 1} = \sigma n^{-1} Z v$. 
From (\ref{eq:repn}) we have 
\begin{equation}
  \label{eq:u2dist}
  \| u \|^2 \stackrel{\mathcal{D}}{=} \sigma^2 n^{-2} \chi_{(n)}^2
  \chi_{(p)}^2 \stackrel{a.s.}{\rightarrow} \sigma^2 c,
\end{equation}
as $p/n \rightarrow c \in [0,\infty)$.

If $\rho$ is any fixed vector in $\mathbb{R}^p$, it follows from
(\ref{eq:repn}) that
\begin{displaymath}
  \tau = \tau(p) =
     \rho^T u / \| \rho \| \| u \| \stackrel{\mathcal{D}}{=} U_{p,1},
\end{displaymath}
the distribution of the first component of $U_p$.
It is well known that $U_1^2 \sim \text{Beta}(1/2, (p-1)/2)$, so that
$E U_1^2 = p^{-1}$ and 
$\text{Var} U_1^2 \leq 2p^{-2}$.
From this it follows that 
\begin{equation}
  \label{eq:taucge}
  \tau(p) \stackrel{a.s.}{\rightarrow} 0, \qquad p \rightarrow \infty.
\end{equation}

\bigskip

(d) Let $u^j = \sigma n^{-1} Z v^{j}$ be the vectors appearing in the
definition of $B^j$ for $1 \leq j \leq m$. We will show that a.s. 
\begin{equation}
  \label{eq:ujasbd}
  \lim_{n \rightarrow \infty} \sup_j \| u^j \| < c_0
\end{equation}
(the constant $c_0 = 2 \sigma (1 + \sqrt c)$ would do).

\begin{proof}
Since 
$$\| u^j \|^2 = \sigma^2 n^{-2} v^{jT} Z^T Z v^{j} $$
we have 
\begin{equation}
   \label{eq:uj2}
   \sup_j \| u^j \|^2 \leq 
   \sigma^2 n^{-1} \lambda_{max}(Z Z^T) \sup_j \| v^j \|^2/n.
\end{equation}
From (\ref{eq:geman}), it follows that w.p. 1, ultimately
\begin{equation}
  \label{eq:lamas1}
\lambda_{max}(Z Z^T)/n \leq 2 (1 + \sqrt c)^2.
\end{equation}

The squared lengths $\| v^j \|^2$ follow independent $\chi_{(n)}^2$
laws. Since from (\ref{eq:deze}) there exists $c_1$ for which
$P \{ \chi_{(n)}^2 \geq 2n \} \leq e^{-c_1 n}$
for $n \geq n_0$, it follows that
\begin{displaymath}
  P \{ \sup \| v^j \|^2/n > 2 \} \leq p e^{-c_1 n}
\end{displaymath}
and so w.p. $1$ it is ultimately true that 
\begin{equation}
  \label{eq:vjas}
  \sup_j \| v^j \|^2/n \leq 2.
\end{equation}
Substituting (\ref{eq:lamas1}) and (\ref{eq:vjas}) into (\ref{eq:uj2}),
we recover (\ref{eq:ujasbd}).
\end{proof}

\bigskip
\bigskip

\subsection{Upper Bounds: Proof of Theorems \ref{thm:upperbd} and
  \ref{th:multi-cpt}}
\label{sec:upper-bounds:-proof}


Instead of working directly with the sample covariance matrix $S$, we
consider $S^* = S - \sigma^2 I_p.$
It is apparent that $S^*$ has the same eigenvectors as $S$.
We decompose $S^* = R + E$, where $R$ is given by (\ref{eq:Rdecomp})
and has spectrum
\begin{displaymath}
  \lambda(R) = \{ \| \rho^1\|^2, \cdots, \| \rho^m \|^2, 0 \}.
\end{displaymath}
The perturbation matrix 
$E = A + B + C,$
where $A$ and $B$ refer to the sums in (\ref{eq:sabc}).
\begin{proposition}
\label{prop:pertbd}
Assume that multicomponent model (\ref{eq:mcm}) holds, along with
assumptions (a) - (d).
For any $\epsilon >0$, if $p, n
\rightarrow \infty, p/n \rightarrow c$, then almost surely
\begin{equation}
\limsup \|E\|_2 \le \sigma \sqrt{c} \sum
\varrho_j + \sigma^2 (c + 2\sqrt{c}).
\label{eq:estar2}
\end{equation}
\end{proposition}

\begin{proof}
We will obtain a bound in the form
\begin{displaymath}
  \| E \|_2 \leq E_n(\omega)
   = A_n(\omega) + B_n(\omega) + C_n(\omega),
\end{displaymath}
where the $A_n, B_n$ and $C_n$ will be given below. We have shown
explicitly the dependence on $\omega$ to emphasize that these
quantities are random. Finally we show that the a.s. limit of
$E_n(\omega)$ is the right side of (\ref{eq:estar2}). 

\textit{$A$ term.} \ Introduce symmetric matrices $2 \tilde A^{jk} =
  A^{jk} + A^{kj} = v^{jk}_s (\rho^j \rho^{kT} + \rho^k \rho^{jT}).$
  Since $\rho^j$ and $\rho^k$ are orthogonal, (\ref{eq:ranktwo})
  implies that
\begin{displaymath}
  \| \tilde A^{jk} \|_2 \leq |v_s^{jk} | \, \| \rho^j \| \, \| \rho^k \|,
\end{displaymath}
and so
\begin{displaymath}
  \| \sum_{j, k} A^{jk} \|_2 \leq
  \max_{j, k} |v_s^{jk} | \Bigl( \sum_j \| \rho^j \| \Bigr)^2 
  =: A_n(\omega).
\end{displaymath}
The $v_s^{jk}$ are  entries of a scaled and recentered $W_m(n,I)$
matrix, and so by (\ref{eq:maxcge}), the maximum converges almost
surely to $0$. Since $\sum_j \| \rho^j \|
\rightarrow \sum  \varrho_j < \infty,$ it follows that the
$A_n$-term converges to zero a.s.

\textit{$B$ term.} Applying (\ref{eq:ranktwo}) to the definition
(\ref{eq:abc}) of $B^j$, we have
\begin{displaymath}
  \| B^j \|_2 \leq X_n(j) = (1 + |\tau_j|) \| \rho^j \| \| u^j \| 
  \stackrel{a.s.}{\rightarrow} \sigma \sqrt c \varrho_j
\end{displaymath}
where $\tau_j = \rho^{jT} u^j / \| \rho^j \| \| u^j \|$ and 
$u^j = \sigma n^{-1} Z v^{j},$ and the convergence follows from
(\ref{eq:u2dist}) and (\ref{eq:taucge}).

Since $|\tau_j| \leq 1$ and using (\ref{eq:ujasbd}), we have a.s. that
for $n > n(\omega)$,
\begin{displaymath}
  X_n(j) \leq Y_n(j) := 2 c_0 \| \rho^j \| \rightarrow 2 c_0 \varrho_j.
\end{displaymath}
The norm convergence (\ref{eq:rhocge}) implies that $\sum_j Y_n(j)
\rightarrow 2 c_0 \sum \varrho_j$ and so it follows from the version 
of the dominated convergence theorem due to \cite{prat60} that 
\begin{displaymath}
  \sum \| B^j \|_2 \leq \sum_j X_n(j) =: B_n(\omega)
\stackrel{a.s.}{\rightarrow} \sigma \sqrt{c} \sum
  \varrho_j. 
\end{displaymath}


\textit{$C$ term.} \ Using (\ref{eq:geman}),
\begin{displaymath} 
   C_n(\omega) = 
  \| C \|_2 = \lambda_{max} (C) \stackrel{a.s.}{\rightarrow} \sigma^2 (c + 2 \sqrt c).
\end{displaymath}

\end{proof}

\bigskip
\bigskip

\textbf{Proof of Theorem \ref{th:multi-cpt}}
[Theorem \ref{thm:upperbd} is a special case.]
We apply the perturbation theorem with $A=R$ and $E=A+B+C$.
The separation between the principal eigenvalue of $R$ and the
remaining ones is 
\begin{displaymath}
  \delta_n = \rho_1^2(n) - \rho_2^2(n) \rightarrow \rho_1^2 - \rho_2^2,
\end{displaymath}
while from Proposition \ref{prop:pertbd} we have the bound
\begin{displaymath}
  \| E \|_2 \leq E_n(\omega) \stackrel{a.s.}{\rightarrow}
  \sigma \sqrt c \varrho_+ + \sigma^2(c + 2 \sqrt c).
\end{displaymath}
Consequently, if 
\begin{displaymath}
  4 \sigma \sqrt c \varrho_+ + \sigma^2(c + 2 \sqrt c) 
   \leq \varrho_1^2 - \varrho_2^2,
\end{displaymath}
then
\begin{displaymath}
  \limsup_{n \rightarrow \infty} \ \text{dist} (\hat \rho^1, \rho^1) 
  \leq \Omega(\rho,c;\sigma),
\end{displaymath}
where
\begin{displaymath}
  \Omega(\rho,c;\sigma) = 4 \sigma \sqrt c [ \varrho_+ + \sigma(\sqrt
  c + 2)]/ (\varrho_1^2 - \varrho_2^2).
\end{displaymath}

\bigskip
\bigskip


\subsection{Lower Bounds: Proof of Theorem \ref{th:lowerbd}}
\label{sec:lower-bounds:-proof}


We begin with a heuristic outline of the proof. 
We write $S$ in the form $ D + B$, introducing
\begin{displaymath}
  D = (1 + v_s) \rho \rho^T + \sigma^2 n^{-1} Z Z^T,
\end{displaymath}
while, as before, $B = \rho u^T + u \rho^T$ and $u = \sigma n^{-1} Z
v$. 

A symmetry trick plays a major role: write $S_- = D-B$
and let $\hat\rho_-$ be  the principal unit eigenvector for $S_-$.

The argument makes precise the following chain of remarks, which are
made plausible by reference to Figure \ref{fig:inconsist}.

\begin{figure}[htb]
     \begin{center}
       \leavevmode
        \centerline{\includegraphics[ width = .28\textwidth,angle =
    0]{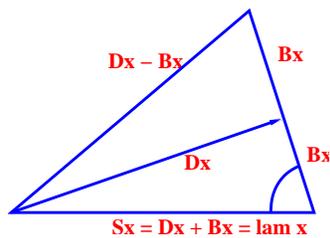}}
\caption{ \small \texttt{Needs caption}, with $x \leftarrow \hat \rho,
lam \leftarrow \hat \lambda$}
       \label{fig:inconsist}
     \end{center}
\end{figure}

(i) $B\hat \rho$ is nearly orthogonal to $D \hat \rho + B \hat \rho =
S \hat \rho = \hat \lambda \hat \rho$.

(ii) the side length $\| B \hat \rho \|$ is bounded away from zero,
when $c > 0$.

(iii) the angle between $\hat \rho$ and $S_- \hat \rho$ is ``large'',
i.e. bounded away from zero.

(iv) the angle between $\hat \rho$ and $\hat \rho_-$ is ``large'' 
[this follows from Lemma \ref{lem:perturb} applied to $M = S_-$.]

(v) and finally, the angle between $\hat \rho$ and $\rho$ must be
``large'', due to the equality in distribution of $\hat \rho$ and
$\hat \rho_-$. 

\medskip
Getting down to details, we will establish (i)-(iii) under the
assumption that $\hat \rho$ is close to $\rho$.
Specifically, we show that given $\delta > 0$ small, there exists 
$\alpha(\delta) = \alpha(\delta; \sigma,c)>0$ such that
w.p. $\rightarrow 1$,
\begin{displaymath}
  \angle(\hat\rho, \rho) \leq \delta \quad \Rightarrow \quad
  \angle(\hat\rho, S_- \hat\rho) \geq \alpha(\delta).
\end{displaymath}

Let $N_\delta = \{ x \in \mathbb{R}^p: \angle(x,\rho) \leq \delta \}$
be the (two-sided) cone of vectors making angle at most $\delta$ with
$x$.
We show that on $N_\delta$, both

(ii') $\| B x \|$ is bounded below (see (\ref{eq:bxlower})), and

(i') $Bx$ is nearly orthogonal to $x$ (see (\ref{eq:cosbd}). 

For convenience in this proof, we may take $\| \rho \| = 1$.
Write $x \in N_{\delta_1}$ in the form
\begin{equation}
  \label{eq:xcoord}
  x = (\cos \delta) \rho + (\sin \delta) \eta,
 \qquad \eta \perp \rho, \| \eta \| = 1, 0 \leq \delta \leq \delta_1.
\end{equation}
Since $B\rho = (u^T \rho) \rho + u$ and $B\eta = (u^T \eta) \rho$, we
find that
\begin{equation}
  \label{eq:bx}
  Bx = (\cos \delta)u + [(\cos \delta)(u^T\rho) + (\sin \delta)(u^T
  \eta)] \rho.
\end{equation}
Denote the second right side term by $r$: clearly $\| r \| \leq | u^T
\rho| + (\sin \delta) \| u \|$, and so, uniformly on $N_\delta$,
\begin{displaymath}
  \| Bx \| \geq ( \cos \delta - \sin \delta) \| u \| - | u^T \rho|.
\end{displaymath}
Since both $\|u\| \rightarrow \sigma \sqrt{c}$ and $u^T \rho
\rightarrow 0$ a.s., we conclude that w.p. $\rightarrow 1$,
\begin{equation}
  \label{eq:bxlower}
  \inf_{N_\delta} \|  Bx \| \geq \hf \sigma \sqrt{c} \cos \delta.
\end{equation}

Turning to the angle between $x$ and $Bx$, we find from
(\ref{eq:xcoord}) and (\ref{eq:bx}) that
\begin{displaymath}
  x^T Bx = 2 (\cos^2 \delta) (\rho^T u) + 2 \cos \delta \sin \delta
  (u^T \eta),
\end{displaymath}
and so, uniformly over $N_\delta$,
\begin{displaymath}
  | x^T Bx | \leq 2 \cos^2 \delta | \rho^T u| + ( \sin 2 \delta) \| u
    \|.
\end{displaymath}

Consequently, using $\| x \|=1$ and (\ref{eq:bxlower}),
w.p. $\rightarrow 1$, and for $\delta < \pi/4$, say,
\begin{equation}
  \label{eq:cosbd}
  | \cos \angle(Bx,x)| 
  = \frac{|x^T Bx|}{\|x\| \| Bx \|}
  \leq \frac{2 \sigma \sqrt{c} \sin 2\delta}{\hf \sigma \sqrt{c} \cos
  \delta} 
  \leq c_2 \delta.
\end{equation}

Now return to Figure \ref{fig:inconsist}. 
As a prelude to step (iii), we establish a lower bound for $\alpha =
\angle ( \hat \rho,D \hat \rho).$
Applying the sine rule (\ref{eq:sinerule}) to $\xi = D \hat \rho$ and
$\eta = \hat \lambda \hat \rho = D \hat \rho + B \hat \rho$, we obtain
\begin{equation}
  \label{eq:Drhotri}
  \sin \angle ( D \hat \rho, \hat \rho) =
    \frac{\| B \hat \rho \|}{ \| D \hat \rho \|} \sin \angle ( B \hat
    \rho, \hat \rho ).
\end{equation}
On the assumption that $\hat \rho \in N_\delta$, bound
(\ref{eq:cosbd}) yields 
\begin{displaymath}
  \sin \angle ( B \hat \rho, \hat \rho ) \geq \sin (\pi/2 - c_3 \delta),
\end{displaymath}
and (\ref{eq:bxlower}) implies that 
\begin{displaymath}
  \| B \hat \rho \| \geq \hf \sigma \sqrt c \cos \delta.
\end{displaymath}
On the other hand, since $\| \hat \rho \| = 1$,
\begin{align*}
  \| D \hat \rho \| & \leq \| D \| \leq 1 + v_s + 
      \sigma^2 \lambda_{max}(n^{-1} Z  Z^T) \\
 & \leq [1 + \sigma^2(1+\sqrt{c})^2](1+o(1)), \notag
\end{align*}
w.p. 1 for large $n$.

Combining the last three bounds into (\ref{eq:Drhotri}) shows that
there exists a positive $\alpha(\delta; \sigma,c)$ such that if $\hat
\rho \in N_\delta$, then w.p. 1 for large $n$,
\begin{displaymath}
  \sin \alpha \geq \sin \alpha(\delta; \sigma,c) > 0.
\end{displaymath}

Returning to Figure \ref{fig:inconsist}, 
consider $\angle( D\hat \rho + B\hat \rho, D\hat \rho -  B\hat \rho) = \alpha
+ \gamma$. Since $\beta \geq \pi/2 - c_3 \delta$, we clearly have 
$\alpha+\gamma \leq \pi - \beta \leq \pi/2 + c_3 \delta$ and hence
\begin{align*}
  \angle( D\hat \rho  + B\hat \rho , D\hat \rho  -  B\hat \rho ) 
  & = \min \{ \alpha + \gamma, \pi -  \alpha - \gamma \} \\
  & \geq  \min \{ \alpha, \pi/2 - c_3 \delta \}.
\end{align*}

In particular, with $\delta \leq \delta_0(\sigma,c)$,
\begin{displaymath}
  \angle( \hat \rho, S_- \hat \rho ) 
  \geq \min \{ \alpha(\delta), \pi/2 - c_3 \delta \} = \alpha(\delta),
\end{displaymath}
which is our step (iii).
As mentioned earlier, Lemma \ref{lem:perturb} applied to $M = S_-$
entails that $\angle ( \hat \rho, \hat \rho_-) \geq (1/3) \alpha
(\delta)$. 
For the rest of the proof, we write $\hat \rho_+$ for $\hat \rho$.
To summarize to this point, we have shown that if 
$\angle (\hat \rho_+, \rho) \leq \delta$, then
w.p. $\rightarrow 1$, 
\begin{equation}
  \label{eq:alphalower}
  \angle ( \hat \rho_+, \hat \rho_-) \geq (1/3) \alpha (\delta).
\end{equation}

Note that $S$ and $S_-$ have the same distribution:
viewed as functions of random terms $Z$ and $v$:
\begin{displaymath}
  S_-(Z,v) = S_+(Z,-v).
\end{displaymath}
We call an event $\mathcal{A}$ symmetric if $(Z,v) \in \mathcal{A}$
iff $(Z,-v) \in \mathcal{A}$. For such symmetric events
\begin{displaymath}
  E [ \angle(\hat\rho_+, \rho), \mathcal{A} ] =
  E [ \angle(\hat\rho_-, \rho), \mathcal{A} ].
\end{displaymath}
From this and the triangle inequality for angles 
\begin{displaymath}
  \angle(\hat\rho_+, \rho) + \angle(\rho, \hat\rho_-) \geq 
  \angle(\hat\rho_+, \hat\rho_-),
\end{displaymath}
it follows that
\begin{equation}
  \label{eq:symmetry}
  E [ \angle(\hat\rho_+, \rho), \mathcal{A} ] \geq
  \hf E [ \angle(\hat\rho_+, \hat\rho_-), \mathcal{A} ]
\end{equation}
Hence
\begin{displaymath}
   E[ \angle(\hat\rho_+, \rho) ] \geq
   E[ \angle(\hat\rho_+, \rho), \mathcal{A}^c ] +
   \hf E [ \angle(\hat\rho_+, \hat\rho_-), \mathcal{A} ].
\end{displaymath}

By the symmetry of the distributions, conclusion (\ref{eq:alphalower})
is also obtained w.p. $\rightarrow 1$ if $\angle (\hat \rho_-, \rho)
\leq \delta$.  
Consequently, letting $\mathcal{A}$ refer to the symmetric event
$ \mathcal{A}_\delta = \{
\angle(\hat\rho_+,\rho) \leq \delta \} \cup \{ \angle(\hat\rho_-,\rho)
\leq \delta \}$, we have
\begin{align*}
  E[ \angle(\hat\rho_+, \rho) ] 
  & \geq \delta P( \mathcal{A}_\delta^c) + \hf E[ \angle( \hat\rho_+,
  \hat\rho_-), \mathcal{A}_\delta ] \\
  & \geq \min \{ \delta, \alpha(\delta)/6 \} (1 + o(1)).
\end{align*}
This completes the proof of Theorem \ref{th:lowerbd}. The lower bound
proof for Theorem \ref{th:multi-cpt} proceeds similarly, but is
omitted -- for some extra detail, see \citet{lu02}.


\subsection{Proof of Theorem \ref{th:fe-fi}.}
\label{sec:theorem4}

We may assume, without loss of generality, that $\sigma_1^2 \geq
\sigma_2^2 \geq \cdots \geq \sigma_p^2.$

\medskip
\textit{False inclusion.} For any fixed constant $t$,
\begin{displaymath}
  \hat \sigma_i^2 \geq t \, \ \text{for} \, \ i = 1, \ldots, k \ \
  \text{and} \ \  \hat \sigma_l^2 < t \\
  \ \ \Rightarrow \ \ \hat \sigma_l^2 < \hat \sigma_{(k)}^2.
\end{displaymath}
This threshold device leads to bounds on error probabilities using
only marginal distributions. For example, consider false inclusion of
variable $l$:
\begin{displaymath}
  P \{ \hat \sigma_l^2 \geq \hat \sigma_{(k)}^2 \} 
  \leq \sum_{i=1}^k P \{ \hat \sigma_i^2 < t \} + P \{ \hat \sigma_l^2
  \geq t \}. 
\end{displaymath}
Write $\bar M_n$ for a $\chi^2_{(n)}/n$ variate, and note from
(\ref{eq:chisqassn}) that $\hat \sigma_\nu^2 \sim \sigma_\nu^2 \bar
M_n$.
Set $t = \sigma_k^2(1-\epsilon_n)$ for a value of $\epsilon_n$ to be
determined. 
Since $\sigma_i^2 \geq \sigma_k^2$ and $\sigma_l^2 \leq
\sigma_k^2(1-\alpha_n)$, we arrive at 
\begin{align*}
  P \{ \hat \sigma_l^2 \geq \hat \sigma_{(k)}^2 \} 
  & \leq k P \{ \bar M_n < 1 - \epsilon_n \} + P \Bigl\{ \bar M_n \geq 
   \frac{1- \epsilon_n}{1 - \alpha_n} \Bigr\} \\
 & \leq k \exp \Bigl\{ - \frac{n \epsilon_n^2}{4} \Bigr\} + 
          \exp \Bigl\{ -\frac{3n}{16} 
                  \Bigl(\frac{\alpha_n - \epsilon_n}{1-\alpha_n}\Bigr)^2 
 \Bigr\} 
\end{align*}
using large deviation bound (\ref{eq:chi2lo}). With the choice 
$\epsilon_n = \sqrt 3 \alpha_n / (2 + \sqrt 3)$, both exponents are
bounded above by $-b(\gamma) \log n$, and so
$P \{ FI \} \leq p (k+1) n^{-b(\gamma)}$.

\medskip
\textit{False exclusion.}  The argument is similar, starting with the
remark that for any fixed $t$,
\begin{displaymath}
  \hat \sigma_i^2 \leq t \, \ \text{for} \, \ i \geq k, i \neq l \ \
  \text{and} \ \  \hat \sigma_l^2 \geq t \\
  \ \ \Rightarrow \ \ \hat \sigma_l^2 \geq \hat \sigma_{(k)}^2.
\end{displaymath}

Consequently, if we set $t = \sigma_k^2(1+\epsilon_n)$ and use
$\sigma_l^2 \geq \sigma_k^2 (1 + \alpha_n)$, we get
\begin{align*}
  P \{ \hat \sigma_l^2 < \hat \sigma_{(k)}^2 \} 
  & \leq \sum_{i\geq k} P \{ \hat \sigma_i^2 > t \} + P \{ \hat \sigma_l^2
  < t \} \\ 
  & \leq (p-1) P \{ \bar M_n > 1 + \epsilon_n \} + 
      P \bigl\{ \bar M_n > \frac{1+ \epsilon_n}{1 + \alpha_n} \bigr\} \\
  & \leq (p-1) \exp \Bigl\{ - \frac{3n\epsilon_n^2}{16}  \Bigr\} + 
          \exp \Bigl\{ -\frac{n}{4}
          \Bigl(\frac{\alpha_n - \epsilon_n}{1+\alpha_n}\Bigr)^2 
               \Bigr\}, 
 \end{align*}
this time using (\ref{eq:chi2hi}).

The bound $P\{ FE \} \leq pk n^{-b(\gamma)} + k e^{- b(\gamma) (1 - 2
  \alpha_n) \log n}$  follows on setting
$\epsilon_n = 2 \alpha_n/(2 + \sqrt 3)$ and noting that $(1 +
\alpha_n)^{-2} \geq 1 - 2 \alpha_n$ .

For numerical bounds, we may collect the preceding bounds in the form
\begin{equation}
  \label{eq:p-bound}
  P(FE \cup FI) \leq 
    [pk + (p-1)(k-1)] e^{-b(\gamma) \log n} 
    + p e^{-b(\gamma) \log n /(1-\alpha_n)^2}
    + (k-1) e^{-b(\gamma) \log n /(1+\alpha_n)^2}.
\end{equation}

\subsection{Proof of Theorem \ref{th:consistency}}
\label{sec:proof-theorem-consist}

\textit{Outline}. \ 
Recall that $\gamma_n = \gamma (n^{-1} \log n)^{1/2},$  that the
selected subset of variables $\hat I$ is defined by
\begin{displaymath}
  \hat I = \{ \nu: \hat \sigma_\nu^2 \geq \sigma^2(1+ \gamma_n) \}
\end{displaymath}
and that the estimated principal eigenvector based on $\hat I$ is
written $\hat \rho_I$. We set
\begin{displaymath}
  \rho_I = ( \rho_\nu : \nu \in \hat I ),
\end{displaymath}
and will use the triangle inequality 
$d(\hat \rho_I, \rho) \leq d( \hat \rho_I, \rho_I) + d( \rho_I, \rho)$
to show that $\hat \rho_I \rightarrow \rho$. There are three main steps.

(i) Construct deterministic sets of indices
\begin{displaymath}
  I_n^\pm = \{ \nu: \rho_\nu^2 \geq \sigma^2 a_\mp \gamma_n \}
\end{displaymath}
which bracket $\hat I$ almost surely as $n \rightarrow \infty$:
\begin{equation}
  \label{eq:bracket}
  I_n^- \subset \hat I \subset I_n^+ \qquad \qquad \text{w.p.} \ 1.
\end{equation}

(ii) the uniform sparsity, combined with $\hat I^c \subset I_n^{-c},$ is
used to show that
\begin{displaymath}
  d(\rho_I,\rho) \stackrel{a.s.}{\rightarrow} 0.
\end{displaymath}

(iii) the containment $\hat I \subset I_n^+$, combined with $| I_n^+ |
= o(n)$ shows via methods similar to Theorem \ref{th:fe-fi} that 
\begin{displaymath}
  d( \hat \rho_I, \rho_I ) \stackrel{a.s.}{\rightarrow} 0.
\end{displaymath}


\medskip
\textit{Details.} \ 
Step (i). We first obtain a bound on the cardinality of $I_n^\pm$
using the uniform sparsity conditions (\ref{eq:unif-sparsity}).
Since $|\rho|_{(\nu)} \leq C \nu^{-1/q}$
\begin{align*}
  |I_n^\pm| & \leq |\{ \nu ~:~ C^2 \nu^{-2/q} \geq \sigma^2 a_\mp
  \gamma_n \}|, \\
            & \leq C^q/(\sigma^2 a_\mp \gamma_n  )^{q/2} = o(n^{1/2}).
\end{align*}

Turning to the bracketing relations (\ref{eq:bracket}), we first
remark that $\hat \sigma_\nu^2 \stackrel{\mathcal{D}}{=} \sigma_\nu^2
\chi_{(n)}^2/n$, and when $\nu \in I_n^\pm$,
\begin{displaymath}
  \sigma_\nu^2 = \sigma^2(1 + \rho_\nu^2/\sigma^2) \geq \sigma^2 (1+
  a_\mp \gamma_n).
\end{displaymath}
Using the definitions of $\hat I$ and writing $\bar M_n$ for 
a random variable with the
distribution of $\chi^2_{(n)}/n$, we have
\begin{align*}
  P_n^- = P( I_n^- \nsubseteq \hat I ) & \leq 
        \sum_{\nu \in I_n^-} P \{ \hat \sigma_\nu^2 < \sigma^2(1+
        \gamma_n) \}      \\
   & \leq |I_n^- | P \{ \bar M_n < (1+\gamma_n)/(1+ a_+ \gamma_n ) \}.
\end{align*}
We apply (\ref{eq:chi2lo}) with $\epsilon_n = (a_+ -1) \gamma_n/(1+
a_+ \gamma_n)$ and for $n$ large and $\gamma'$ slightly smaller than
$\gamma^2$, 
\begin{displaymath}
  n \epsilon_n^2 > (a_+ -1)^2 \gamma' \log n,
\end{displaymath}
so that
\begin{displaymath}
  P_n^- \leq c n^{1/2} \exp \{ - n \epsilon_n^2 /4 \} \leq c n^{1/2 -
  \gamma_+^{''}} 
\end{displaymath}
with $\gamma_+^{''} = (a_+-1)^2 \gamma'/4.$
If $\sqrt{\gamma} \geq 12,$ then $\gamma_+^{''} \geq 3$ for suitable
$a_+ > 2.$

The argument for the other inclusion is analogous:
\begin{align*}
  P_n^+ = P( \hat I \nsubseteq  I_n^+ ) & \leq 
        \sum_{\nu \notin I_n^+} P \{ \hat \sigma_\nu^2 \geq \sigma^2(1+
        \gamma_n) \}      \\
   & \leq p P \{ \bar M_n \geq (1+\gamma_n)/(1+ a_- \gamma_n )
        \} \\
   & \leq p n^{ - \gamma_-^{''}},
\end{align*}
with $\gamma_-^{''} = 3(1-a_-)^2 \gamma'/16$ so long as $n$ is large
enough.
If $\sqrt{\gamma} \geq 12$, then $\gamma_-^{''} > 2$ for suitable $a_-
< 1 - \sqrt{8/9}$.

By a Borel-Cantelli argument, (\ref{eq:bracket}) follows from the
bounds on $P_n^-$ and $P_n^+$.

\bigskip

Step (ii). For $n > n(\omega)$ we have $I_n^- \subset \hat I$ and so
\begin{displaymath}
  \| \rho_I - \rho \|^2 = \sum_{\nu \notin \hat I } \rho_\nu^2
  \leq \sum_{I_n^{-c}} \rho_\nu^2.
\end{displaymath}
When $\nu \in I_n^{-c}$, we have by definition
\begin{displaymath}
  \rho_\nu^2(n) < \sigma^2 a_+ \gamma \sqrt{ n^{-1} \log n } :=
  \epsilon_n^2, 
\end{displaymath}
say, while the uniform sparsity condition entails
\begin{displaymath}
  |\rho|_{(\nu)}^2 \leq C^2 \nu^{-2/q}.
\end{displaymath}

Putting these together, and defining $s_* = s_*(n)$ as the
solution of the equation $C s^{-1/q} = \epsilon_n$, we obtain 
\begin{align*}
  \sum_{I_n^{-c}} \rho_\nu^2
  & \leq \sum_\nu \epsilon_n^2 \wedge \rho_\nu^2 \\
  & = \sum_\nu \epsilon_n^2 \wedge |\rho|_{(\nu)}^2 \\ 
  & \leq \sum_\nu \epsilon_n^2 \wedge C^2 \nu^{-2/q}  \\ 
  & \leq  \int_0^\infty \epsilon_n^2 \wedge C^2 s^{-2/q} ds \\
  & =  s_* \epsilon_n^2 + q(2-q)^{-1} C^2 s_*^{1-2/q}  \\ 
  & = [2/(2-q)] C^q \epsilon_n^{2-q} \rightarrow 0
\end{align*}
as $n \rightarrow \infty.$ 

\bigskip

Step (iii).  We adopt the abbreviations 
\begin{align*}
   u_I    & = ( u_\nu    ~:~ \nu \in \hat I ), \\ 
   Z_I    & = ( z_{\nu i}    ~:~ \nu \in \hat I, i = 1, \ldots, n ), \\ 
   S_I    & = ( S_{\nu \nu'} ~:~ \nu, \nu' \in \hat I).
\end{align*}
As in the proof of Theorem \ref{th:multi-cpt}, we consider
$ S^*_I =  S_I - \sigma^2 I_{\hat k} =  \rho_I 
\rho^T_I +  E_I$ and note that the perturbation term has the
decomposition
\begin{displaymath}
   E_I = v_s  \rho_I  \rho^T_I +  \rho_I  u^T_I 
            +  u_I  \rho^T_I 
            + \sigma^2 ( n^{-1}  Z_I  Z^T_I - I),
\end{displaymath}
so that
\begin{displaymath}
  \|  E_I \|_2 \leq v_s \|  \rho_I \|_2^2 
        + 2 \|  \rho_I \|_2 \|  u_I \|_2 
        + \sigma^2 [ \lambda_{max}(n^{-1}  Z_I  Z^T_I) - 1 ].
\end{displaymath}

Consider the first term on the right side.
Since $\|  \rho_I- \rho \|_2 \stackrel{a.s.}{\rightarrow} 0$ from
step (ii), it follows that $\|  \rho_I \|_2
\stackrel{a.s.}{\rightarrow} \| \rho \|$.
As before $v_s \stackrel{a.s.}{\rightarrow} 0$, and so the first term
is asymptotically negligible.

Let $Z_{I^+} = ( z_{\nu i}    ~:~ \nu \in  I^+_n, i = 1, \ldots, n
)$ and $ u_{I^+}  = ( u_\nu    ~:~ \nu \in I_n^+ )$.
On the event $\Omega_n = \{ \hat I \subset I_n^+ \}$, we have
\begin{displaymath}
  \|  u_I \| \leq \| u_{I^+} \|
\end{displaymath}
and setting $k_+ = |I_n^+|$, by the same arguments as led to 
(\ref{eq:u2dist}), we have
\begin{displaymath}
  \| u_{I^+} \|^2 \stackrel{\mathcal{D}}{=}
  \sigma^2 (k_+/n) (\chi_{(n)}^2/n) (\chi_{(k_+)}^2/k_+) 
  \stackrel{a.s.}{\rightarrow} 0,
\end{displaymath}
since $k_+ = o(n)$ from step (i).

Finally, since on the event $\Omega_n$, the matrix $Z_{I^+}$ contains
$ Z_I$, along with some additional rows, it follows that
\begin{displaymath}
  \lambda_{max}( n^{-1}  Z_I  Z^T_I -I) 
  \leq \lambda_{max}( n^{-1} Z_{I^+} Z_{I^+}^T -I)
  \stackrel{a.s.}{\rightarrow} 0
\end{displaymath}
by (\ref{eq:geman}), again since $k_+ =o(n)$.
Combining the previous bounds, we conclude that $ \|  E_I \|_2
\rightarrow 0$. 

The separation $ \delta_n = \|  \rho_I \|_2^2 \rightarrow \|
\rho \|_2^2 > 0$ and so by the perturbation bound
\begin{displaymath}
  \text{dist}( \hat \rho_I, \rho_I) 
   \leq (4/ \delta_n) \|  E_I \|_2
   \stackrel{a.s.}{\rightarrow} 0. 
\end{displaymath}

\textbf{Acknowledgements.}  The authors are grateful for helpful
comments from Debashis Paul and the participants at the Functional
Data Analysis meeting at Gainesville, FL. January 9-11, 2003.
This work was supported in part by
grants NSF DMS 0072661 and NIH EB R01 EB001988.

\bibliographystyle{agsm}

\end{document}